\documentclass[oneside,french,english,12pt]{amsart}
\usepackage[latin9]{inputenc}
\usepackage{float}
\usepackage{mathtools}
\usepackage{amstext}
\usepackage{amsthm}
\usepackage{amssymb}
\usepackage[backref=page,colorlinks,citecolor=blue,bookmarks=true]{hyperref}\hypersetup{linkcolor=blue,  citecolor=blue, urlcolor=blue}

\DeclareRobustCommand{\subtitle}[1]{\\#1}

\makeatletter

\providecommand{\tabularnewline}{\\}
\floatstyle{ruled}
\newfloat{algorithm}{tbp}{loa}
\providecommand{\algorithmname}{Algorithm}
\floatname{algorithm}{\protect\algorithmname}

\numberwithin{equation}{section}
\numberwithin{figure}{section}
\theoremstyle{plain}
\newtheorem{thm}{\protect\theoremname}[section]
\newtheorem*{thm*}{\protect\theoremname}
\theoremstyle{definition}
\newtheorem{problem}[thm]{\protect\problemname}
\newtheorem*{problem*}{Problem}
\theoremstyle{remark}
\newtheorem*{rem*}{\protect\remarkname}
\theoremstyle{remark}
\newtheorem{rem}[thm]{\protect\remarkname}
\newenvironment{lyxcode}
	{\par\begin{list}{}{
		\setlength{\rightmargin}{\leftmargin}
		\setlength{\listparindent}{0pt}% needed for AMS classes
		\raggedright
		\setlength{\itemsep}{0pt}
		\setlength{\parsep}{0pt}
		\normalfont\ttfamily}%
	 \item[]}
	{\end{list}}
\theoremstyle{definition}
\newtheorem{defn}[thm]{\protect\definitionname}
\theoremstyle{plain}
\newtheorem{prop}[thm]{\protect\propositionname}
\theoremstyle{plain}
\newtheorem{fact}[thm]{\protect\factname}
\theoremstyle{plain}
\newtheorem{lem}[thm]{\protect\lemmaname}

\usepackage{amsrefs}
\usepackage{algorithmic}
\usepackage{fullpage}
\usepackage{microtype}
\usepackage{tikz}
\usepackage{changepage}

\floatname{algorithm}{Algorithm}

\let\originalleft\left
\let\originalright\right
\renewcommand{\left}{\mathopen{}\mathclose\bgroup\originalleft}
\renewcommand{\right}{\aftergroup\egroup\originalright}

\date{}

\makeatother

\usepackage{babel}
\makeatletter
\addto\extrasfrench{%
   \providecommand{\fg}{\ifdim\lastskip>\z@\unskip\fi~\frqq}%
}

\makeatother
\addto\captionsenglish{\renewcommand{\definitionname}{Definition}}
\addto\captionsenglish{\renewcommand{\factname}{Fact}}
\addto\captionsenglish{\renewcommand{\lemmaname}{Lemma}}
\addto\captionsenglish{\renewcommand{\problemname}{Problem}}
\addto\captionsenglish{\renewcommand{\propositionname}{Proposition}}
\addto\captionsenglish{\renewcommand{\remarkname}{Remark}}
\addto\captionsenglish{\renewcommand{\theoremname}{Theorem}}
\addto\captionsfrench{\renewcommand{\algorithmname}{Algorithme}}
\addto\captionsfrench{\renewcommand{\definitionname}{Définition}}
\addto\captionsfrench{\renewcommand{\factname}{Fait}}
\addto\captionsfrench{\renewcommand{\lemmaname}{Lemme}}
\addto\captionsfrench{\renewcommand{\problemname}{Problème}}
\addto\captionsfrench{\renewcommand{\propositionname}{Proposition}}
\addto\captionsfrench{\renewcommand{\remarkname}{Remarque}}
\addto\captionsfrench{\renewcommand{\theoremname}{Théorème}}
\providecommand{\definitionname}{Definition}
\providecommand{\factname}{Fact}
\providecommand{\lemmaname}{Lemma}
\providecommand{\problemname}{Problem}
\providecommand{\propositionname}{Proposition}
\providecommand{\remarkname}{Remark}
\providecommand{\theoremname}{Theorem}

\DeclareMathOperator{\SL}{SL}
\DeclareMathOperator{\PSL}{PSL}
\DeclareMathOperator{\Sym}{Sym}
\DeclareMathOperator{\defect}{def}
\let\U\relax
\DeclareMathOperator{\U}{U}
\let\index\relax
\DeclareMathOperator{\index}{ind}
\DeclareMathOperator{\id}{id}
\DeclareMathOperator{\op}{op}
\DeclareMathOperator{\hs}{hs}
\DeclareMathOperator{\strict}{strict}
\DeclareMathOperator{\T}{(T)}
\DeclareMathOperator{\ptau}{(\tau)}
\let\hom\relax
\DeclareMathOperator{\hom}{Hom}
\let\OE\relax
\DeclareMathOperator{\OE}{OL}
\DeclareMathOperator{\trieq}{(\triangle)} % \ensuremath{\triangle}
\DeclareMathOperator{\Lin}{Lin}
\DeclareMathOperator{\tr}{tr}

\newcommand{\eps}{\varepsilon}

\newcommand{\CC}{\mathbb{C}}

\newcommand{\NN}{\mathbb{N}}

\newcommand{\QQ}{\mathbb{Q}}
\newcommand{\RR}{\mathbb{R}}

\newcommand{\ZZ}{\mathbb{Z}}

\newcommand{\calC}{\mathcal{C}}

\newcommand{\calF}{\mathcal{F}}
\newcommand{\calG}{\mathcal{G}}
\newcommand{\calH}{\mathcal{H}}

\newcommand{\calO}{\mathcal{O}}
\newcommand{\calP}{\mathcal{P}}

\newcommand{\calS}{\mathcal{S}}

\newcommand{\dm}{\mathop{dm}}
\newcommand{\intint}{{\int\!\!\int}}
\newcommand{\rightedge}{{\hspace{-0.05em}\longrightarrow\hspace{-0.05em}}}%

\begin{document}
\title[Stability of Approximate group actions: \subtitle{Uniform and probabilistic}]{Stability of Approximate group actions: \subtitle{Uniform and probabilistic}}
\author[O.\ Becker]{Oren Becker}
\address{Oren Becker\hfill\break
	Department of Pure Mathematics and Mathematical Statistics\hfill\break
	Centre for Mathematical Sciences\hfill\break
	Wilberforce Road, Cambridge CB3 0WA, United Kingdom}
\email{oren.becker@gmail.com}

\author[M.\ Chapman]{Michael Chapman}
\address{Michael Chapman\hfill\break
	Einstein Institute of Mathematics\hfill\break
	The Hebrew University, Jerusalem 91904, Israel.}
\email{michael.chapman@mail.huji.ac.il}

\begin{abstract}
We prove that every uniform approximate homomorphism from a discrete
amenable group into a symmetric group is uniformly close to a homomorphism
into a slightly larger symmetric group. That is, amenable groups are
uniformly flexibly stable in permutations. This answers affirmatively
a question of Kun and Thom and a slight variation of a question of
Lubotzky. We also give a negative answer to Lubotzky's original question
by showing that the group $\ZZ$ is not uniformly strictly stable.
Furthermore, we show that $\SL_{r}\left(\ZZ\right)$, $r\geq3$, is
uniformly flexibly stable, but the free group $F_{r}$, $r\geq2$,
is not. We define and investigate a probabilistic variant of uniform
stability that has an application to property testing.
\end{abstract}

\maketitle

\section{Introduction}

In 1940, Ulam asked the following general question, usually referred
to as \emph{Ulam's stability problem} \cite{Ulam,Hyers}: given two
groups $\Gamma$ and $G$ and an approximate homomorphism $f\colon\Gamma\to G$,
is $f$ close to a homomorphism? The answer depends on the groups
$\Gamma$ and $G$ as well as the chosen notions of an approximate
homomorphism and proximity between functions.

The following theorem of Kazhdan tackles a particular instance
of Ulam's problem: is every approximate unitary representation of
a group close in operator norm to a unitary representation?
\begin{thm*}
[Kazhdan 1982, \cite{Kazhdan}] \label{thm:Kazhdan}Let $f\colon\Gamma\rightarrow\U\left(\calH\right)$
be a function from an amenable group $\Gamma$ into the group $\U\left(\calH\right)$
of unitary operators on the Hilbert space $\calH$. Take $\delta<1/200$
such that $\|f\left(\gamma_{1}\gamma_{2}\right)-f\left(\gamma_{1}\right)f\left(\gamma_{2}\right)\|_{\op}\leq\delta$
for all $\gamma_{1},\gamma_{2}\in\Gamma$. Then, there is a group
homomorphism $h\colon\Gamma\rightarrow\U\left(\calH\right)$ such
that $\|h\left(\gamma\right)-f\left(\gamma\right)\|_{\op}\leq2\delta$
for every $\gamma\in\Gamma$.
\end{thm*}

The present paper tackles a similar problem: is every approximate
action on a finite set close to an action? That is, we replace the
unitary groups in Kazhdan's theorem by finite symmetric groups. We
refer to this version of Ulam stability as \emph{stability in permutations},
and make use of the normalized Hamming metric on $\Sym\left(n\right)$:
\[
d^{H}\left(\sigma,\tau\right)=\frac{1}{n}\left|\left\{ x\in\left[n\right]\mid\sigma\left(x\right)\ne\tau\left(x\right)\right\} \right|\qquad\forall\sigma,\tau\in\Sym\left(n\right)\,\,\text{,}
\]
where $\left[n\right]=\left\{ 1,\dotsc,n\right\} $. For a group $\Gamma$
and a function $f\colon\Gamma\rightarrow\Sym\left(n\right)$, we define
the \emph{uniform local defect} of $f$ to be
\[
\defect_{\infty}\left(f\right)=\sup\left\{ d^{H}\left(f\left(\gamma_{1}\gamma_{2}\right),f\left(\gamma_{1}\right)f\left(\gamma_{2}\right)\right)\mid\gamma_{1},\gamma_{2}\in\Gamma\right\} \,\,\text{.}
\]

A basic result of Glebsky and Rivera studies stability in
permutations when the domain group is finite.
\begin{thm*}
[Glebsky and Rivera 2009, \cite{GlebskyRivera}]\label{thm:Glebsky-Rivera} Let $\Gamma$ be
a finite group and $f\colon\Gamma\rightarrow\Sym\left(n\right)$ a
function, $n\in\NN$. Then, there is a group homomorphism $h\colon\Gamma\rightarrow\Sym\left(n\right)$
such that $d^{H}\left(h\left(\gamma\right),f\left(\gamma\right)\right)\leq C\defect_{\infty}\left(f\right)$
for every $\gamma\in\Gamma$, where $C$ depends only on the group
$\Gamma$ (but not on $n$).
\end{thm*}

In Section \ref{subsec:intro-framework}, we recall the terminology
of group theoretic stability and two types of a uniformly
stable group: strict and flexible. The above theorem says that
each finite group is uniformly strictly stable.
In the spirit of Kazhdan's Theorem, it is natural to ask
whether the same is true for infinite amenable
groups. The first test case, as raised by Alex Lubotzky, is the following.
\begin{problem*}
[Lubotzky 2018]\label{prob:lubotzky} Does the theorem of Glebsky and Rivera
hold when the finite group $\Gamma$ is replaced by $\ZZ$?
\end{problem*}

In Section \ref{sec:non-strict} we give a negative answer to Lubotzky's
question: $\ZZ$ is not uniformly strictly stable. In fact, we prove the following stronger result.
\begin{thm}
\label{thm:intro-not-strictly-stable}Let $\Gamma$ be a group
that has a transitive action on $\left[n\right]$.
Then, there is a function $f\colon\Gamma\rightarrow\Sym\left(n-1\right)$
such that $\defect_{\infty}\left(f\right)\leq\frac{2}{n-1}$,
but for every homomorphism $h\colon\Gamma\rightarrow\Sym\left(n-1\right)$
there is $\gamma\in\Gamma$ such that $d^{H}\left(f\left(\gamma\right),h\left(\gamma\right)\right)\geq\frac{1}{2}-\frac{1}{n-1}$.
\end{thm}

The theorem implies a negative answer to Lubotzky's question, not
only for $\Gamma=\ZZ$, but for every group that has finite quotients
of unbounded cardinality. Furthermore, by considering the action of
a finite group on itself by left multiplication, Theorem \ref{thm:intro-not-strictly-stable}
implies that in the theorem of Glebsky and Rivera  the dependence
of the constant $C$ on the group $\Gamma$ is essential.

These negative results might be discouraging at first. However, another
stability theorem, proved by Gowers and Hatami \cite{GowersHatami}
and generalized by De Chiffre, Ozawa and Thom \cite{DOT}, leads us
in the right direction. Recall that the normalized Hilbert--Schmidt
norm on $\U\left(n\right)$ is given by $\|A\|_{\hs}=\left(\frac{1}{n}\tr\left(A^{*}A\right)\right)^{1/2}$.
\begin{thm*}
[Gowers and Hatami 2017, \cite{GowersHatami}. De Chiffre, Ozawa and Thom 2019, \cite{DOT}] \label{thm:gowers-hatami} Let
$\Gamma$ be a discrete amenable group and $f\colon\Gamma\rightarrow\U\left(n\right)$
a function. Take $\delta>0$ such that $\|f\left(\gamma_{1}\gamma_{2}\right)-f\left(\gamma_{1}\right)f\left(\gamma_{2}\right)\|_{\hs}\leq\delta$
for all $\gamma_{1},\gamma_{2}\in\Gamma$. Then, there is
a representation $h\colon\Gamma\rightarrow\U\left(N\right)$ and an
isometry $T\colon\CC^{n}\rightarrow\CC^{N}$ such that $\|h\left(\gamma\right)-T^{*}f\left(\gamma\right)T\|_{\hs}\le211\delta$
for every $\gamma\in\Gamma$,
where $n\leq N\leq\left(1+2500\delta^{2}\right)n$.
\end{thm*}

It is also shown in \cite{GowersHatami} and \cite{DOT} that allowing
the increase in dimension from $n$ to $N$ is essential. Similarly,
in the context of functions into $\Sym\left(n\right)$, we shall allow
a controlled increase in the number of points $n$. We refer to this
approach as \emph{flexibility in the number of points. }It was proven
to be effective and necessary in many instances \cite{BowenBurton,BeckerLubotzky,LLM}.
We extend the definition of the normalized Hamming metric to measure
distances between elements of symmetric groups of different cardinalities.
For $n\leq N$, $\sigma\in\Sym\left(n\right)$ and $\tau\in\Sym\left(N\right)$,
define: 
\[
d^{H}\left(\sigma,\tau\right)=d^{H}\left(\tau,\sigma\right)=\frac{1}{N}\left(\left|\left\{ x\in\left[n\right]\mid\sigma\left(x\right)\neq\tau\left(x\right)\right\} \right|+\left(N-n\right)\right)\,\,\text{.}
\]
The \emph{uniform distance} between two functions $f\colon\Gamma\rightarrow\Sym\left(n\right)$
and $h\colon\Gamma\rightarrow\Sym\left(N\right)$ is 
\[
d_{\infty}\left(f,h\right)=d_{\infty}\left(h,f\right)=\sup\left\{ d^{H}\left(f\left(\gamma\right),h\left(\gamma\right)\right)\mid\gamma\in\Gamma\right\} \,\,\text{.}
\]
By Lemma \ref{lem:triangle-inequality}, $d^{H}$ satisfies the triangle
inequality and thus it is a metric on the disjoint union $\coprod_{n=1}^{\infty}\Sym\left(n\right)$.

The following question was asked by Kun and Thom.
\begin{problem*}
[Kun and Thom 2019, {\cite[Remark~4.3]{KunThom}}] \label{prob:kun-thom} Take a function
$f\colon\Gamma\rightarrow\Sym\left(n\right)$, where $\Gamma$ is
a finite group and $n\in\NN$.
%Write $\delta=\defect_{\infty}\left(f\right)$.
Is there a homomorphism $h\colon\Gamma\rightarrow\Sym\left(N\right)$
such that $d_{\infty}\left(h,f\right)\leq\eps$
and $n\leq N\leq\left(1+\eps\right)n$,
where $\eps$ depends only on $\defect_{\infty}\left(f\right)$
and tends to zero as $\defect_{\infty}\left(f\right)$ tends to zero.
\end{problem*}

\begin{rem*}
For $N\geq n$, $\sigma\in\Sym\left(n\right)$ and $\tau\in\Sym\left(N\right)$,
our definition of $d^{H}$ forces $d^{H}\left(\sigma,\tau\right)\geq1-\frac{n}{N}$.
Hence, the condition $d_{\infty}\left(h,f\right)\leq\eps$ in the
above problem implies that $N\leq\frac{n}{1-\eps}\le\left(1+2\eps\right)n$
whenever $\eps\leq1/2$. Thus, the condition $N\leq\left(1+\eps\right)n$
in the problem statement is redundant. We state it solely for the
sake of emphasis.
\end{rem*}
The problem of Kun and Thom is not solved by the theorem of Glebsky and Rivera
because the latter provides $\eps$ that depends on the domain group $\Gamma$
and not only on $\defect_{\infty}\left(f\right)$.
As already mentioned, the dependency on $\Gamma$ is
essential when flexibility in the number of points is not allowed.
Our main result in the present paper is the next theorem. It gives
an affirmative answer to Kun and Thom's problem and to a flexible
variant of Lubotzky's problem. In fact, it only assumes
that $\Gamma$ is amenable (rather than finite or infinite cyclic),
and provides an explicit $\eps$ which is linear in $\defect_{\infty}\left(f\right)$.
\begin{thm}[Amenable groups are uniformly flexibly stable]
\label{thm:intro-amenable-positive}Let $\Gamma$ be a discrete amenable
group and $f\colon\Gamma\rightarrow\Sym\left(n\right)$ a function,
$n\in\NN$. Then, there is
a homomorphism $h\colon\Gamma\rightarrow\Sym\left(N\right)$ such
that $d_{\infty}\left(h,f\right)\leq2039\defect_{\infty}\left(f\right)$
and
$n\leq N\leq\left(1+1218\defect_{\infty}\left(f\right)\right)n$.
%and a homomorphism $h\colon\Gamma\rightarrow\Sym\left(N\right)$ such
%that $d_{\infty}\left(h,f\right)\leq2039\defect_{\infty}\left(f\right)$.
\end{thm}

Theorem \ref{thm:intro-amenable-positive} has the following useful
corollary (see Theorem \ref{thm:almost-trivial-on-finite-index}):
for a finite-index normal subgroup $\Delta$ of $\Gamma$ and a function
$f\colon\Gamma\rightarrow\Sym\left(n\right)$, there is a homomorphism
$h\colon\Gamma\rightarrow\Sym\left(N\right)$, $N\geq n$, such that
$d_{\infty}\left(h,f\right)\leq C\cdot\left(\defect_{\infty}\left(f\right)+\sup\left\{ d^{H}\left(f\left(\gamma\right),\id\right)\mid\gamma\in\Delta\right\} \right)$,
where $C$ is a universal constant. Notably, the bound on $d_{\infty}\left(h,f\right)$
does not depend on the index $\left[\Gamma\colon\Delta\right]$. In
fact, it suffices to assume that $\Gamma/\Delta$ is a discrete amenable
(rather than finite) group, and thus Theorem \ref{thm:almost-trivial-on-finite-index}
is a strong form of Theorem \ref{thm:intro-amenable-positive}
(up to the constants).

Using Theorems \ref{thm:intro-amenable-positive} and \ref{thm:almost-trivial-on-finite-index},
and following the method of \cite[Section 5]{BOT}, we deduce the
following theorem.
\begin{thm}[$\SL_{r}\left(\calO_{K}\right)$ is uniformly flexibly stable]
\label{thm:intro-SLr}Let $f\colon\SL_{r}\left(\calO_{K}\right)\rightarrow\Sym\left(n\right)$
be a function, where $n\in\NN$, $r\geq3$ and $\calO_{K}$ is the
ring of integers of a number field $K$.
Then there is
a homomorphism $h\colon\SL_{r}\left(\calO_{K}\right)\rightarrow\Sym\left(N\right)$
such that $d_{\infty}\left(h,f\right)\leq C\defect_{\infty}\left(f\right)$
and $n\leq N\leq\left(1+C\defect_{\infty}\left(f\right)\right)n$,
where $C$ depends only on $r$.
\end{thm}

The next theorem shows that for some groups
not every approximate homomorphism is close to a homomorphism,
even when flexibility in the number of
points is allowed.
\begin{thm}[Nonabelian free groups are not uniformly flexibly stable]
\label{thm: Free_gps_are_not_FUSP-1}Let $\Gamma$ be a group that
surjects onto the free group $F_{2}$ of rank $2$. Then, there is
a sequence of functions $\left(f_{k}\right)_{k=1}^{\infty}$, $f_{k}\colon\Gamma\rightarrow\Sym\left(n_{k}\right)$,
$n_{k}\overset{k\rightarrow\infty}{\longrightarrow}\infty$, such
that $\defect_{\infty}\left(f_{k}\right)\leq\frac{2}{k}$, but $d_{\infty}\left(h_{k},f_{k}\right)\geq1-\frac{5}{k}$
for every homomorphism $h_{k}\colon\Gamma\rightarrow\Sym\left(N_{k}\right)$
for all $N_{k}\geq n_{k}$.
\end{thm}

\begin{rem*}
Theorem \ref{thm: Free_gps_are_not_FUSP-1} is analogous to a result
of Burger, Ozawa and Thom \cite[Proposition 3.3]{BOT} that says that
for $r\geq2$, an approximate homomorphism $F_{r}\rightarrow\U\left(n\right)$
need not be close to a homomorphism. More generally, the same is true
whenever $F_{r}$ is replaced by a group $\Gamma$ such that the comparison
map $H_{b}^{2}\left(\Gamma,\RR\right)\rightarrow H^{2}\left(\Gamma,\RR\right)$
is noninjective \cite[Corollary 3.5]{BOT}. This condition holds when
$\Gamma$ is a nonelementary word-hyperbolic group \cite{EpsteinFujiwara}
(and was known earlier in the special case $\Gamma=F_{r}$ \cite{Brooks81}).

The discrete nature of $\Sym\left(n\right)$ makes it difficult to
prove Theorem \ref{thm: Free_gps_are_not_FUSP-1} by following the
lines of the proofs of \cite[Proposition 3.3]{BOT} or \cite[Corollary 3.5]{BOT}.
Our proof does not use cohomological methods. We do draw inspiration
from Rolli's construction of nontrivial quasimorphisms from $F_{r}$ to $\RR$
\cite{Rolli}.
%(the existence of a nontrivial quasimorphism $\Gamma\rightarrow\RR$
%is equivalent to the noninjectivity of the comparison map $H_{b}^{2}\left(\Gamma,\RR\right)\rightarrow H^{2}\left(\Gamma,\RR\right)$).
\end{rem*}

In light of the success of \cite{EpsteinFujiwara} in generalizing the
cohomological result of \cite{Brooks81} from nonabelian free groups to
nonelementary word-hyperbolic groups, we pose the following open problem
that asks whether Theorem \ref{thm: Free_gps_are_not_FUSP-1}
can be generalized as well.
\begin{problem}[Are all nonelementary hyperbolic groups uniformly flexibly instable?]
Let $\Gamma$ be a nonelementary word-hyperbolic group.
Is there
a sequence of functions $\left(f_{k}\right)_{k=1}^{\infty}$,
$f_{k}\colon\Gamma\rightarrow\Sym\left(n_{k}\right)$,
$n_k\in\NN$,
such that
$\defect_{\infty}\left(f_{k}\right)\overset{k\rightarrow\infty}{\longrightarrow}0$,
but $d_{\infty}\left(f_{k},h_{k}\right)\geq C$ for every 
homomorphism $h_{k}\colon\Gamma\rightarrow\Sym\left(N_{k}\right)$,
$N_{k}\geq n_{k}$,
where $C>0$ does not depend on $k$?
\end{problem}
\ 

\subsection{A framework for stability}
\label{subsec:intro-framework}
The above results can be formulated in a single framework.
We consider the following objects.
\begin{itemize}
\item Two classes of groups: $\calC$ and $\calG$.
\item The full class of functions $\calF=\left\{ f\colon\Gamma\rightarrow G\mid\Gamma\in\calC,G\in\calG\right\} $.
\item A function $\defect\colon\calF\rightarrow\RR_{\geq0}$, called the
\emph{local defect}.
\item A distance function $d\colon\calF\times\calF\rightarrow\RR_{\geq0}\cup\left\{ \infty\right\} $.
\end{itemize}
The \emph{global defect} of $f\in\calF$ w.r.t. $d$ is
\[
D_{d}\left(f\right)=\inf\left\{ d\left(f,h\right)\mid\text{\ensuremath{h\in\calF} is a group homomorphism}\right\} \,\,\text{.}
\]
We say that the class $\calC$ of groups is \emph{$\left(\calG,\defect,d\right)$-stable}
if
$D_{d}\left(f\right)\leq R\left(\defect\left(f\right)\right)$
for every $f\in\calF$,
where $R\left(\delta\right)\overset{\delta\rightarrow0}{\longrightarrow}0$.
We say that $\calC$ is stable \emph{with linear rate} if  $R\left(\delta\right)\leq C\delta$ for a universal constant $C>0$.
A group $\Gamma$ is \emph{$\left(\calG,\defect,d\right)$-stable}
if the same is true for the class $\left\{ \Gamma\right\} $.

All of the results presented so far can be formulated using this
framework. In the context of stability in permutations, the role
of $\calG$ is taken by $\calS=\left\{ \Sym\left(n\right)\right\} _{n=1}^{\infty}$.
The role of $\defect$ is taken by $\defect_{\infty}$. Let $f,h\in\calF$.
We extend $d_{\infty}$ by setting $d_{\infty}\left(f,h\right)=\infty$
if the domains of $f$ and $h$ are different. We define $d_{\infty}^{\strict}\left(f,h\right)=d_{\infty}\left(f,h\right)$
if $f$ and $h$ have the same domain and range, and $d_{\infty}^{\strict}\left(f,h\right)=\infty$
otherwise. A class $\calC$ of groups is \emph{uniformly flexibly
stable in permutations} if it is $\left(\calS,\defect_{\infty},d_{\infty}\right)$-stable.
The class $\calC$ is \emph{uniformly strictly stable in permutations}
if it is $\left(\calS,\defect_{\infty},d_{\infty}^{\strict}\right)$-stable.
From now on, we omit the phrase ``in permutations''.

The results presented thus far about functions into symmetric groups are
summarized in the following table, where the numbers on the right
are theorem numbers or references (in brackets).
\begin{center}
\begin{tabular}{lll}
\noalign{\vskip\doublerulesep}
\textbf{Uniformly strictly stable:} & each finite group & \cite{GlebskyRivera}\tabularnewline
\noalign{\vskip\doublerulesep}
\textbf{Not uniformly strictly stable:} & the class of finite groups & \ref{thm:intro-not-strictly-stable}\tabularnewline
\noalign{\vskip\doublerulesep}
 & each group with unbounded finite quotients & \ref{thm:intro-not-strictly-stable}\tabularnewline
\noalign{\vskip\doublerulesep}
\textbf{Uniformly flexibly stable:} & the class of discrete amenable groups & \ref{thm:intro-amenable-positive}\tabularnewline
\noalign{\vskip\doublerulesep}
 & each $\SL_{r}\left(\ZZ\right)$, $r\geq3$ & \ref{thm:intro-SLr}\tabularnewline
\noalign{\vskip\doublerulesep}
\textbf{Not uniformly flexibly stable:} & each group that surjects onto $F_{2}$ & \ref{thm: Free_gps_are_not_FUSP-1}\tabularnewline
\end{tabular}
\par\end{center}
\vspace{1em}
In fact, \cite{GlebskyRivera} and Theorems \ref{thm:intro-amenable-positive} and \ref{thm:intro-SLr}  prove stability with linear rate.
\begin{rem*}
[Pointwise Stability]The related notion of \emph{pointwise stability}
has recently been under heavy investigation \cite{BeckerLubotzky,BLT,BeckerMosheiff,BowenBurton,DGLT,ESS,GlebskyRivera,HadwinShulman,Ioana,LLM,LevitLubotzky1,LevitLubotzky2,LubotzkyOppenheim,ThomICM,Zheng}.
A \emph{stability challenge} for $\Gamma$ is a sequence of functions
$\left(f_{k}\right)_{k=1}^{\infty}$, $f_{k}\colon\Gamma\rightarrow\Sym\left(n_{k}\right)$,
$n_{k}\in\NN$, such that $d^{H}\left(f_{k}\left(\gamma_{1}\gamma_{2}\right),f_{k}\left(\gamma_{1}\right)f_{k}\left(\gamma_{2}\right)\right)$
tends to zero as $k\rightarrow\infty$ for all $\gamma_{1},\gamma_{2}\in\Gamma$.
A \emph{solution} for $\left(f_{k}\right)_{k=1}^{\infty}$ is a sequence
of homomorphisms $\left(h_{k}\right)_{k=1}^{\infty}$, $h_{k}\colon\Gamma\rightarrow\Sym\left(N_{k}\right)$,
$N_{k}\geq n_{k}$, such that $d^{H}\left(f_{k}\left(\gamma\right),h_{k}\left(\gamma\right)\right)$
tends to zero as $k\rightarrow\infty$ for each $\gamma\in\Gamma$.
The group $\Gamma$ is \emph{pointwise stable in permutations} if
every stability challenge for $\Gamma$ has a solution.

For finitely presented groups, pointwise stability can be formalized
using the above framework if we allow each object of $\calC$ to be
a group with a fixed presentation (rather than just a group). Proximity
between functions is measured with respect to the images of the generators,
and the local defect is defined using the relators.

In the context of pointwise stability, the strict and flexible versions
are equivalent when the group $\Gamma$ is amenable \cite[Lemma~3.2(1)]{Ioana}.
The main result
of \cite{BLT} provides a useful equivalent condition for pointwise
stability among amenable groups. In particular, and as has been known
previously \cite{ArzhantsevaPaunescu,GlebskyRivera}, some amenable
groups are pointwise stable and some are not (in contrast with the flexible uniform case, as shown by Theorem \ref{thm:intro-amenable-positive}).
For $r\geq3$, the group $\SL_{r}\ZZ$ is not strictly stable both
in the pointwise sense \cite{BeckerLubotzky} and in the uniform sense
(by Theorem \ref{thm:intro-not-strictly-stable}). By Theorem \ref{thm:intro-SLr},
$\SL_{r}\ZZ$ is uniformly flexibly stable, but it is not known whether
it is pointwise flexibly stable. By \cite{BowenBurton}, if one can
find $r\geq5$ such that $\PSL_{r}\ZZ$ is pointwise flexibly stable,
it would solve a long-standing open problem by showing that not all
groups are sofic.
\end{rem*} 
%\ 

\subsection{Probabilistic stability and homomorphism testing}

Let $\Gamma$ be a discrete amenable group. Fix a finitely-additive
measure $m$ on $\Gamma$ that is either left or right invariant.
The \emph{mean local defect} of a function $f\colon\Gamma\rightarrow\Sym\left(n\right)$
is given by
\[
\defect_{1}\left(f\right)=\intint d^{H}\left(f\left(\gamma_{1}\gamma_{2}\right),f\left(\gamma_{1}\right)f\left(\gamma_{2}\right)\right)\dm\left(\gamma_{1}\right)\dm\left(\gamma_{2}\right)\,\,\text{.}
\]
For $N\geq n$, $f\colon\Gamma\rightarrow\Sym\left(n\right)$ and
$h\colon\Gamma\rightarrow\Sym\left(N\right)$, the \emph{mean distance}
between $f$ and $h$ is
\[
d_{1}\left(f,h\right)=d_{1}\left(h,f\right)=\int d^{H}\left(f\left(\gamma\right),h\left(\gamma\right)\right)\dm\left(\gamma\right)\,\,\text{.}
\]

It is well known that every discrete amenable group $\Gamma$ admits
a finitely-additive measure that is simultaneously left, right and
inverse invariant (see Section \ref{subsec:amenable}). Henceforth,
we fix such a measure for each discrete amenable group. We prove that
the class of discrete amenable groups is \emph{probabilistically flexibly
stable}. That is, it is $\left(\calS,\defect_{1},d_{1}\right)$-stable.
More precisely, we prove the following analogue of Theorem \ref{thm:intro-amenable-positive}.
\begin{thm}
[Amenable groups are probabilistically flexibly stable]
\label{thm:intro-amenable-on-average-positive}Let $\Gamma$ be a
discrete amenable group and $f\colon\Gamma\rightarrow\Sym\left(n\right)$
a function, $n\in\NN$.
Then, there is
a homomorphism $h\colon\Gamma\rightarrow\Sym\left(N\right)$
such that
$d_{1}\left(h,f\right)\leq2913\defect_{1}\left(f\right)$ and
$n\leq N\leq\left(1+1740\defect_{1}\left(f\right)\right)n$.
\end{thm}

Theorem \ref{thm:intro-amenable-on-average-positive} has an application
to property testing in the case where the group $\Gamma$ is finite.
We begin by recalling a generalized version of the Blum--Luby--Rubinfeld
Theorem \cite{BLR} on homomorphism testing.
\begin{thm*} [Blum, Luby and Rubinfeld 1990, {\cite[Theorem 2.3]{Goldreich}}] 
\label{thm:BLR} Let $f\colon\Gamma\rightarrow G$
be a function between finite groups. Let 
\[
\delta=\frac{1}{\left|\Gamma\times\Gamma\right|}\left|\left\{ \left(\gamma_{1},\gamma_{2}\right)\in\Gamma\times\Gamma\mid f\left(\gamma_{1}\gamma_{2}\right)\neq f\left(\gamma_{1}\right)f\left(\gamma_{2}\right)\right\} \right|\,\,\text{.}
\]
If $\delta<\frac{1}{6}$, then there is a homomorphism $h\colon\Gamma\rightarrow G$
such that 
\[
\frac{1}{\left|\Gamma\right|}\left|\left\{ \gamma\in\Gamma\mid h\left(\gamma\right)\neq f\left(\gamma\right)\right\} \right|\leq2\delta\,\,\text{.}
\]
\end{thm*}

Take $f$, $\Gamma$ and $G$ as in the theorem. The finite
group $\Gamma$ should be thought of as very large. Assume that $f$
is either a homomorphism or disagrees with every homomorphism $\Gamma\rightarrow G$
on many elements of $\Gamma$. The goal in homomorphism testing is
to distinguish between the two cases after reading $f\left(\gamma\right)$
for only a small number of elements $\gamma\in\Gamma$. This is achieved
by the following algorithm.

\begin{algorithm}[H]
\caption{Homomorphism testing}

\begin{raggedright}
\textbf{Input:} A function $f\colon\Gamma\rightarrow G$ between finite
groups
\par\end{raggedright}
\begin{raggedright}
\textbf{Output:} Accept or Reject
\par\end{raggedright}
\begin{lyxcode}
1:~Sample~$\left(\gamma_{1},\gamma_{2}\right)\in\Gamma\times\Gamma$~uniformly~at~random.

2:~If~$f\left(\gamma_{1}\gamma_{2}\right)=f\left(\gamma_{1}\right)f\left(\gamma_{2}\right)$,~return~Accept.

3:~Otherwise,~return~Reject.
\end{lyxcode}
\end{algorithm}

Clearly, if $f$ is a homomorphism then the algorithm accepts. Let
$0<\eps\leq1/3$ and assume that $f$ disagrees with every group homomorphism
$h\colon\Gamma\rightarrow G$ on at least $\eps\left|\Gamma\right|$
elements of $\Gamma$. The Blum--Luby--Rubinfeld Theorem implies that the
algorithm rejects with probability at least $\eps/2$. For
$\alpha>0$, we can amplify the rejection probability to be at least
$1-\alpha$ by running the algorithm for $k=\lceil\log_{1-\eps/2}\alpha\rceil=O\left(\frac{2\log\left(1/\alpha\right)}{\eps}\right)$
independent iterations and accepting if and only if all iterations accept.
Note that $k$ is independent of $\Gamma$ and $G$.

Now consider the scenario where the group $\Gamma$ is still very
large and $G=\Sym\left(n\right)$, where $n$ is also very large.
In this case, reading the permutation $f\left(\gamma\right)$, even
just for a single element $\gamma\in\Gamma$, may be too time consuming.
The following algorithm is tailored for this situation.

\begin{algorithm}[H]
\caption{\foreignlanguage{french}{\label{alg:Testing-of-homomorphisms}\foreignlanguage{english}{Testing
of homomorphisms into $\protect\Sym\left(n\right)$}}}

\begin{raggedright}
\textbf{Input:} A function $f\colon\Gamma\rightarrow\Sym\left(n\right)$,
where $\Gamma$ is a finite group and $n\in\NN$
\par\end{raggedright}
\begin{raggedright}
\textbf{Output:} Accept or Reject
\par\end{raggedright}
\begin{lyxcode}
1:~Sample~$\left(\gamma_{1},\gamma_{2},x\right)\in\Gamma\times\Gamma\times\left[n\right]$~uniformly~at~random.

2:~If~$f\left(\gamma_{1}\gamma_{2}\right)\left(x\right)=f\left(\gamma_{1}\right)f\left(\gamma_{2}\right)\left(x\right)$,~return~Accept.

3:~Otherwise,~return~Reject.
\end{lyxcode}
\end{algorithm}

Again, if $f$ is a homomorphism then the algorithm always accepts.
On the other hand, for $0<\eps\leq1$, if $d_{1}\left(f,h\right)\geq\eps$
for every homomorphism $h\colon\Gamma\rightarrow\Sym\left(N\right)$,
$N\geq n$, then the probability that the algorithm rejects is at
least $\frac{\eps}{2913}$. As before, the rejection probability may
be amplified by running the algorithm repeatedly.

Finally, we show that strict stability has the same caveats in the
probabilistic setting as in the uniform setting. That is, our proof
of Theorem \ref{thm:intro-not-strictly-stable} in Section \ref{sec:non-strict}
handles probabilistic stability in addition to uniform stability.
Hence, Algorithm \ref{alg:Testing-of-homomorphisms} is not a good
tester in the strict model.

\begin{rem*}[A recent breakthrough in quantum information theory]
	The recent solution \cite{MIPRE} to Connes' embedding problem
	relies on a probabilistic variant of the Gowers--Hatami Theorem
	(see also \cite{NatarajanVidick, NatarajanVidick2018, VidickBlog}).
	The latter plays a fundamental role in the proof by forcing the
	shared state of nearly optimal provers to be close to a specific
	desired state.
	We would like to know whether there are similar applications
	where Theorem \ref{thm:intro-amenable-on-average-positive}
	can be used instead of the Gowers--Hatami Theorem.
\end{rem*}
\ \\

\subsection{The structure of the paper and some comments on the proofs}

In Section \ref{sec:amenable-stable} we prove Theorems \ref{thm:intro-amenable-positive}
and \ref{thm:intro-amenable-on-average-positive}. The proof takes
a function $f\colon\Gamma\rightarrow\Sym\left(n\right)$ and restricts
each permutation $f\left(\gamma\right)$ to an injective function
$f\left(\gamma\right)\mid_{D\left(\gamma\right)}\colon D\left(\gamma\right)\rightarrow\left[n\right]$,
$D\left(\gamma\right)\subset\left[n\right]$, such that $f\left(\gamma_{1}\right)\mid_{D\left(\gamma_{1}\right)}\circ f\left(\gamma_{2}\right)\mid_{D\left(\gamma_{2}\right)\cap f\left(\gamma_{2}\right)^{-1}D\left(\gamma_{1}\right)}$
and $f\left(\gamma_{1}\gamma_{2}\right)\mid_{D\left(\gamma_{1}\gamma_{2}\right)}$
coincide on the intersection of their domains. We then extend each
$f\left(\gamma\right)\mid_{D\left(\gamma\right)}$ to a permutation
$h\left(\gamma\right)\colon\left[N\right]\rightarrow\left[N\right]$,
$N\geq n$, such that $h\colon\Gamma\rightarrow\Sym\left(N\right)$
is a homomorphism. Amenability is used in the restriction step to
ensure quantitative properties, such as the ratios $\frac{\left|D\left(\gamma\right)\right|}{n}$
being close to $1$ (uniformly in $\gamma$ or on average). After
the restriction step we are left with a purely algebraic structure
(a groupoid). The extension step does not make direct use of amenability.
The proof is related to the proof of homomorphism testing for finite
groups \cite[Theorem 2.3]{Goldreich} in the sense that both employ
majority voting.

In Section \ref{sec:special-linear} we prove a more general
version of Theorem \ref{thm:intro-SLr}. To that end we use Theorem \ref{thm:intro-amenable-positive}
on the subgroups of upper and lower triangular unipotent matrices
of $\SL_{r}$, apply bounded generation \cite{WitteMorris}, and
conclude by using Theorem \ref{thm:almost-trivial-on-finite-index}.
In Section \ref{sec:non-strict} we prove Theorem \ref{thm:intro-not-strictly-stable}
and its probabilistic version. To exclude strict stability
for a group $\Gamma$, we take a homomorphism $f\colon\Gamma\rightarrow\Sym\left( n \right)$
that defines a transitive action and deform it into a function $\hat{f}\colon\Gamma\rightarrow\Sym\left( n-1 \right)$
by bypassing $n$. We invoke
\cite[Proposition 2.4(ii)]{BeckerLubotzky} and some analysis
to show that $\hat{f}$ is far from every homomorphism in the strict
model. In Section \ref{sec:non-flexible} we prove Theorem \ref{thm: Free_gps_are_not_FUSP-1}.
We show that the free group $F_{2}$ is not flexibly uniformly stable
by constructing functions $f_{k}\colon F_{2}\rightarrow\Sym\left(n_{k}\right)$,
with small $\defect_{\infty}\left(f_{k}\right)$, that grossly violate
a group identity that holds in $\Sym\left(N\right)$ for all $N\geq n$.
In Appendix \ref{sec:triangle-ineq} we prove that $d^{H}$ satisfies
the triangle inequality. In Appendix \ref{sec:symmetrization} we
prove an auxiliary result that says that a function $f\colon\Gamma\rightarrow\Sym\left(n\right)$
with small local defect is close to a function that sends $1_{\Gamma}\mapsto\id$
and respects inverses. This result is used in Section \ref{sec:amenable-stable}.
Appendix \ref{sec:symmetrization} handles a more general case, where
$\Sym\left(n\right)$ is replaced by a metric group satisfying a
mild condition. 
\ \\

\subsection*{Acknowledgments}
We are thankful to Irit Dinur, Nati Linial, Alex Lubotzky and Thomas Vidick
for useful discussions and for their comments on this manuscript.

MC is supported in part by the European Research Council (ERC) under the European Union's Horizon 2020 research and innovation program (grant no. 692854)
of Alex Lubotzky. 
\ \\

\section{\label{sec:amenable-stable}Flexible stability of amenable groups}

In this section we prove Theorems \ref{thm:intro-amenable-positive}
and \ref{thm:intro-amenable-on-average-positive}, and deduce Theorem
\ref{thm:almost-trivial-on-finite-index}.

\subsection{\label{subsec:amenable}Preliminaries on amenable groups}

Let $\Gamma$ be a discrete group. Write $\calP\left(\Gamma\right)$
for the power set of $\Gamma$. A \emph{finitely-additive probability
measure} $m$ on $\Gamma$ is function $m\colon\calP\left(\Gamma\right)\rightarrow\left[0,1\right]$
such that $m\left(\Gamma\right)=1$ and $m\left(A\cup B\right)=m\left(A\right)+m\left(B\right)$
whenever $A$ and $B$ are disjoint subsets of $\Gamma$. We call
such an $m$ a \emph{measure} on $\Gamma$ for short. We say that
$m$ is \emph{left invariant} (resp. \emph{right invariant}) if $m\left(\gamma A\right)=m\left(A\right)$
(resp. $m\left(A\gamma\right)=m\left(A\right)$) for every $\gamma\in\Gamma$
and $A\subset\Gamma$. The group $\Gamma$ is \emph{amenable} if it
admits a left-invariant measure. Finite groups are amenable since
we may take $m$ to be the normalized counting measure. In this particular
case, the left-invariant measure is unique. Other basic examples of
amenable groups are abelian groups and, more generally, solvable groups.
Examples of non-amenable groups include free groups on more than one
generator and infinite groups with Property $\T$, such as $\SL_{r}\ZZ$,
$r\geq3$.

Let us give an example of a left-invariant measure $m$ on $\Gamma=\ZZ$.
To each $A\subset\ZZ$ we attach a bounded sequence $\left(a_{t}\right)_{t=1}^{\infty}$,
given by $a_{t}=\frac{\left|A\cap\left[-t,t\right]\right|}{2t+1}$.
If $\left(a_{t}\right)_{t=1}^{\infty}$ converges, we set $m\left(A\right)$
to be its limit. More generally, we set $m\left(A\right)$ to be a
carefully chosen accumulation point of $\left(a_{t}\right)_{t=1}^{\infty}$.
This is made possible by the axiom of choice, and if done carefully,
results in a finitely-additive measure $m$. The left invariance of
$m$ follows from the fact that for a fixed $r\in\ZZ$, the ratio
$\frac{\left|\left[-t,t\right]\triangle \left(r+\left[-t,t\right]\right)\right|}{\left|\left[-t,t\right]\right|}$
tends to zero as $t\rightarrow\infty$, where $\triangle$ denotes symmetric difference.
In other words, $\left(\left[-t,t\right]\right)_{n=1}^{\infty}$
is a \emph{Følner sequence} for $\ZZ$. It is possible to change our
proofs of Theorems \ref{thm:intro-amenable-positive} and \ref{thm:intro-amenable-on-average-positive}
to use $A\mapsto\left|A\cap\left[-t,t\right]\right|$ instead of $m$,
for a carefully chosen large $t$. However, the use of a limit (in
fact, an ultralimit) in the definition of $m$ saves us the effort
of tracking error terms in the course of the proof.

Let $m$ be a measure on $\Gamma$. There is a notion of integration
w.r.t. $m$ of bounded functions $\Gamma\rightarrow\CC$ (see \cite[Section 1.2.2]{DrutuKapovich}).
The integration functional $f\mapsto\int f\dm\colon L^{\infty}\left(\Gamma\right)\rightarrow\CC$
is a positive linear functional such that $\int{\bf 1}_{A}\dm=m\left(A\right)$.
Positivity means that $\int f\dm\geq0$ whenever the image of $f$
is contained in $\RR_{\geq0}$, and linearity means that integration
commutes with finite sums and multiplication by scalars. We may write
$\dm\left(\gamma\right)$ instead of $\dm$ to indicate that $\gamma$
is the variable of integration. If $m$ is left invariant, then $\int f\left(\gamma_{0}\gamma\right)\dm\left(\gamma\right)=\int f\left(\gamma\right)\dm\left(\gamma\right)$
for every $f\in L^{\infty}\left(\Gamma\right)$ and $\gamma_{0}\in\Gamma$.
For $\Omega\subset\Gamma$, write $\int_{\Omega}f\dm$ for $\int f\cdot{\bf 1}_{\Omega}\dm$.

A left-invariant measure $\mu$ on $\Gamma$ gives rise to
$m\colon\calP\left(\Gamma\right)\rightarrow\left[0,1\right]$, given by
\[
m\left(A\right)=\frac{1}{2}\int\left(\mu\left(A\gamma\right)+\mu\left(A^{-1}\gamma\right)\right)\,d\mu\left(\gamma\right)\,\,\text{.}
\]
Then $m$ is a measure on $\Gamma$ and it is \emph{bi-invariant},
that is, simultaneously left and right invariant. Furthermore, it
is \emph{inverse invariant}, that is, $m\left(A\right)=m\left(A^{-1}\right)$
for each $A\subset\Gamma$. Integration with respect to a right-invariant
or inverse-invariant measure has the corresponding invariance property.

Finally, a trivial but useful consequence of left or right invariance
of $m$ is that for a subgroup $H$ of $\Gamma$ we have $\left[\Gamma\colon H\right]=\frac{1}{m\left(H\right)}$,
interpreted as $\infty$ if $m\left(H\right)=0$.
\ \\
\subsection{The setup for the proof of Theorems \ref{thm:intro-amenable-positive}
and \ref{thm:intro-amenable-on-average-positive}}

Fix an amenable group $\Gamma$ with a measure $m$ that is left,
right and inverse invariant. We consider the metrics $d_{\infty}$
and $d_{1}$, and the local defects $\defect_{\infty}$ and $\defect_{1}$,
as defined in the introduction, where $d_{1}$ and $\defect_{1}$
are defined w.r.t. $m$.

\subsubsection{Symmetrization}
\begin{defn}
\label{def:symmetric-function}Let $\Gamma$ and $G$ be groups. A
function $f\colon\Gamma\rightarrow G$ is \emph{symmetric} if $f\left(1_{\Gamma}\right)=1_{G}$
and $f\left(\gamma^{-1}\right)=\left(f\left(\gamma\right)\right)^{-1}$
for all $\gamma\in\Gamma$.
\end{defn}

Appendix \ref{sec:symmetrization} deals with deforming a function
$f\colon\Gamma\rightarrow G$ into a symmetric function, and equips
us with the following proposition.
\begin{prop}
\label{prop:U-symmetrization}Let $\Gamma$ be a group and $f\colon\Gamma\to\Sym\left(n\right)$
a function. Then, there exists a symmetric function $f'\colon\Gamma\to\Sym\left(n\right)$
such that
\begin{align*}
d_{\infty}\left(f,f'\right) & \leq2\defect_{\infty}\left(f\right)\,\,\text{,}\\
d_{1}\left(f,f'\right) & \leq3\defect_{1}\left(f\right)\,\,\text{,}\\
\defect_{\infty}\left(f'\right) & \leq7\defect_{\infty}\left(f\right)\,\,\text{and}\\
\defect_{1}\left(f'\right) & \leq10\defect_{1}\left(f\right)\,\,\text{.}
\end{align*}
\end{prop}

\begin{proof}
The claim is a special case of Proposition \ref{prop:symmetrization-in-general}
(see Lemma \ref{lem:2-torsion-Sym}).
\end{proof}
By virtue of Proposition \ref{prop:U-symmetrization}, Theorems \ref{thm:intro-amenable-positive}
and \ref{thm:intro-amenable-on-average-positive} follow from the
following theorem.
\begin{thm}
\label{thm:amenable-groups-symmetric-function}Let $n\in\NN$ and
let $f\colon\Gamma\rightarrow\Sym\left(n\right)$ be a symmetric function.
Write $\delta_{\infty}=\defect_{\infty}\left(f\right)$ and $\delta_{1}=\defect_{1}\left(f\right)$.
Then, there is $N\geq n$ and a homomorphism $h\colon\Gamma\rightarrow\Sym\left(N\right)$
such that
\[
\text{\ensuremath{N\leq\left(1+174\delta_{\infty}\right)n} and \ensuremath{d_{\infty}\left(h,f\right)\leq291\delta_{\infty}}}
\]
and
\[
\text{\ensuremath{N\leq\left(1+174\delta_{1}\right)n} and \ensuremath{d_{1}\left(h,f\right)\leq291\delta_{1}} .}
\]
\end{thm}

Note that $\delta_{1}\le\delta_{\infty}$ and thus the claim $N\leq\left(1+174\delta_{\infty}\right)n$
follows at once from $N\leq\left(1+174\delta_{1}\right)n$.

\subsubsection{$\Gamma$-graphs}

The proof of Theorem \ref{thm:amenable-groups-symmetric-function}
is based on a graph-theoretic approach.
\begin{defn}
A \emph{$\Gamma$-graph }is a pair $X=\left(V,E\right)$, where $V$
is a set and $E$ is a subset of $V\times\Gamma\times V$. The elements
of $V$ are the \emph{vertices} of $X$ and the elements of $E$ are
the oriented $\Gamma$-labelled \emph{edges} of $X$. We use the notation
$x\overset{\gamma}{\rightedge}y$ to denote the edge $\left(x,\gamma,y\right)$
with \emph{origin} $x$, \emph{destination} $y$ and \emph{label}
$\gamma$. We require that for every $x\in V$ and $\gamma\in\Gamma$
there exists at most one $y\in V$ such that $x\overset{\gamma}{\rightedge}y\in E$.
\end{defn}

With $X$ as above, denote $V\left(X\right)\coloneqq V$ and $E\left(X\right)\coloneqq E$.
We say that $X$ is \emph{finite} if $\left|V\left(X\right)\right|<\infty$.
All $\Gamma$-graphs that appear in our argument are finite in this
sense.
\begin{defn}
\label{def:Function_graph}Let $V$ be a set and take a function $f\colon\Gamma\rightarrow\Sym\left(V\right)$.
The \emph{function graph} $X_{f}$ of $f$ is the $\Gamma$-graph
with vertex set $V$ and edge set
\[
E\left(X_{f}\right)=\left\{ x\overset{\gamma}{\rightedge}f\left(\gamma\right)\left(x\right)\mid x\in V,\gamma\in\Gamma\right\} \,\,\text{.}
\]
\end{defn}

The function $f$ is symmetric if and only if all edges of $X_{f}$
with label $1_{\Gamma}$ are loops and for every edge $x\overset{\gamma}{\rightedge}y\in E(X_f)$,
we have $y\overset{\gamma^{-1}}{\rightedge}x\in E(X_f)$. In this case,
$f$ is a homomorphism if and only if every path in $X_{f}$ of the
form $\overset{\gamma_{1}}{\rightedge}\overset{\gamma_{2}}{\rightedge}\overset{\gamma_{1}^{-1}\gamma_{2}^{-1}}{\rightedge}$
is closed (for every starting vertex). Informally, if $\defect_{1}\left(f\right)$
is small, then almost all of these paths are closed, and the same
is true if $\defect_{\infty}\left(f\right)$ is small since $\defect_{1}\left(f\right)\le\defect_{\infty}\left(f\right)$.

Let $X=\left(V,E\right)$ be a $\Gamma$-graph. We write $x\overset{\gamma}{\rightedge}\in E$
to indicate that there is a vertex $y\in V$ such that $x\overset{\gamma}{\rightedge}y\in E$.
For $\gamma\in\Gamma$, the \emph{domain} of $\gamma$ is $D_{X}\left(\gamma\right)\coloneqq\left\{ x\in V\mid x\overset{\gamma}{\rightedge} \in E\right\} $.
The \emph{set of outgoing labels }from a vertex $x\in V$ is ${\rm \OE}_{X}\left(x\right)\coloneqq\left\{ \gamma\in\Gamma\mid x\overset{\gamma}{\rightedge}\in E\right\} $
and the \emph{out degree} of a vertex $x\in V$ is $\deg_{X}\left(x\right)\coloneqq m\left(\OE\left(x\right)\right)$.
In this way, the measure $m$ enables us to define a useful notion
of a degree in the graph $X$, where a vertex may have infinitely
many incident edges. When the graph $X$ is clear from the context,
we may omit it from the notation in $D_{X}\left(\gamma\right)$ and
$\OE_{X}\left(x\right)$.
\begin{fact}
\label{fact:deg-and-domain}For a $\Gamma$-graph $X=\left(V,E\right)$, the sum of
out degrees is equal to the integral of the cardinalities of domains.
That is,
\[
\sum_{x\in V}\deg_{X}\left(x\right)=\int\left|D_{X}\left(\gamma\right)\right|\dm\left(\gamma\right)\,\,\text{.}
\]
\end{fact}

Fact \ref{fact:deg-and-domain} follows immediately from the definitions
and the basic properties of integration discussed in Section \ref{subsec:amenable}.

Let $X=(V,E)$ and $X'=(V',E')$ be $\Gamma$-graphs. We say that $X$ is a \emph{subgraph} of $X'$ if $V\subseteq V'$ and $E\subseteq E'$.
A function $\varphi\colon V\rightarrow V'$
is a \emph{morphism of $\Gamma$-graphs} from $X$ to $X'$ if $\varphi\left(x_{1}\right)\overset{\gamma}{\rightedge}\varphi\left(x_{2}\right)$
is in $E'$ whenever $x_{1}\overset{\gamma}{\rightedge}x_{2}$
is in $E$. Such a function $\varphi$ is an \emph{embedding}
of $X$ in $X'$ if it is injective.

The heart of the proof of Theorem \ref{thm:amenable-groups-symmetric-function}
lies in the proof of the following proposition.
\begin{prop}
\label{prop:amenable-main}Let $f\colon\Gamma\rightarrow\Sym\left(n\right)$
be a symmetric function, $n\in\NN$. Write $\delta_{\infty}=\defect_{\infty}\left(f\right)$
and $\delta_{1}=\defect_{1}\left(f\right)$, and assume that $\delta_{1}\leq1/78$.
Then there is a subgraph $Z$ of the function graph $X_{f}$, a finite
set $V_{1}$ and a homomorphism $g\colon\Gamma\rightarrow\Sym\left(V_{1}\right)$
such that:\renewcommand{\labelenumi}{\roman{enumi})}
\begin{enumerate}
\item $Z$ embeds in the function graph $X_{g}$.
\item $\left|V\left(Z\right)\right|\geq\left(1-96\delta_{1}\right)n$.
\item For every $\gamma_{0}\in\Gamma$,
\begin{equation}
\left|D_{Z}\left(\gamma_{0}\right)\right|\geq\left(1-117\delta_{\infty}\right)n\label{eq:X0-strong-infty}
\end{equation}
and
\begin{equation}
\int\left|D_{Z}\left(\gamma\right)\right|\dm\left(\gamma\right)\geq\left(1-117\delta_{1}\right)n\,\,\text{.}\label{eq:X0-strong-L1}
\end{equation}
\item $n\leq\left|V_{1}\right|\leq\left(1+78\delta_{1}\right)n$.
\end{enumerate}
\end{prop}

Theorem \ref{thm:amenable-groups-symmetric-function} follows from
Proposition \ref{prop:amenable-main} as follows.
\begin{proof}
[Proof of Theorem \ref{thm:amenable-groups-symmetric-function}]We
are given a function $f\colon\Gamma\rightarrow\Sym\left(n\right)$
and need to define a homomorphism $h\colon\Gamma\rightarrow\Sym\left(N\right)$,
$N\geq n$, such that $h$ and $f$ are close together. If $\delta_{1}>1/78$,
set $N=n$ and let $h\colon\Gamma\rightarrow\Sym\left(N\right)$ be
the trivial homomorphism. Assume henceforth that $\delta_{1}\leq1/78$.
Apply Proposition \ref{prop:amenable-main} to $f$ to obtain a subgraph
$Z$ of the function graph $X_{f}$, a set $V_{1}$ and a homomorphism
$g\colon\Gamma\rightarrow\Sym\left(V_{1}\right)$, such that (i)-(iv)
of Proposition \ref{prop:amenable-main} are satisfied. In particular,
there is an embedding $\varphi\colon V\left(Z\right)\rightarrow V_{1}$
of the $\Gamma$-graph $Z$ in the function graph $X_{g}$. Let $N=\left|V_{1}\right|+\left(n-\left|V\left(Z\right)\right|\right)$.
Then
\[
n\leq N\leq\left(1+\left(78+96\right)\delta_{1}\right)n=\left(1+174\delta_{1}\right)n\,\,\text{.}
\]
Assume without loss of generality that $V_{1}=V\left(Z\right)\coprod\left(\left[N\right]\setminus\left[n\right]\right)$
and that $\varphi$ is the inclusion map. For every $\gamma\in\Gamma$,
define $h\left(\gamma\right)\in\Sym\left(N\right)$ by
\[
h\left(\gamma\right)\left(x\right)=\begin{cases}
g\left(\gamma\right)\left(x\right) & x\in V_{1}\\
x & x\in\left[n\right]\setminus V\left(Z\right)
\end{cases}\qquad\forall x\in\left[N\right]\,\,\text{.}
\]
Let $\gamma\in\Gamma$ and $x\in D_{Z}\left(\gamma\right)$. Then
$g\left(\gamma\right)\left(x\right)=f\left(\gamma\right)\left(x\right)$
since the inclusion map $\varphi\colon V\left(Z\right)\rightarrow V_{1}$
is an embedding of the $\Gamma$-graph $Z$ in $X_{g}$. On the other
hand, $h\left(\gamma\right)\left(x\right)=g\left(\gamma\right)\left(x\right)$
since $x\in V\left(Z\right)\subset V_{1}$. Hence, for every $\gamma_0\in \Gamma$, 
\begin{align}
d^{H}\left(h\left(\gamma_0\right),f\left(\gamma_0\right)\right) & \leq\frac{1}{N}\left(\left(n-\left|D_{Z}\left(\gamma_0\right)\right|\right)+\left(N-n\right)\right)\nonumber \\
 & =1-\frac{n}{N}\frac{\left|D_{Z}\left(\gamma_0\right)\right|}{n}\nonumber \\
 & \leq1-\frac{1}{1+174\delta_{1}}\cdot\frac{\left|D_{Z}\left(\gamma_{0}\right)\right|}{n}\,\,\text{.}\label{eq:fct-from-graph-1}
\end{align}
Hence,
\begin{align*}
d_{\infty}\left(h,f\right) & \leq1-\frac{1-117\delta_{\infty}}{1+174\delta_{\infty}} & \text{by (\ref{eq:fct-from-graph-1}) and (\ref{eq:X0-strong-infty})}\\
 & \leq1-\left(1-174\delta_{\infty}\right)\left(1-117\delta_{\infty}\right)\\
 & \leq291\delta_{\infty}
\end{align*}
and
\begin{align*}
d_{1}\left(h,f\right) & \leq\int\left(1-\frac{1}{1+174\delta_{1}}\cdot\frac{\left|D_{Z}\left(\gamma\right)\right|}{n}\right)\dm\left(\gamma\right) & \text{by (\ref{eq:fct-from-graph-1})}\\
 & \leq1-\left(1-174\delta_{1}\right)\left(1-117\delta_{1}\right) & \text{by (\ref{eq:X0-strong-L1})}\\
 & \leq291\delta_{1}\,\,\text{.}
\end{align*}
\end{proof}
\ \\

\subsection{Proof of Proposition \ref{prop:amenable-main}}

Fix $n\in\NN$ and a symmetric function $f\colon\Gamma\rightarrow\Sym\left(n\right)$.
Write $\delta_{\infty}=\defect_{\infty}\left(f\right)$ and $\delta_{1}=\defect_{1}\left(f\right)$.
We first construct the subgraph $Z$ of the function graph $X_{f}$.
Then, we proceed to define the set $V_{1}$ and the homomorphism $g\colon\Gamma\rightarrow\Sym\left(V_{1}\right)$.
Finally, we show that (i)-(iv) of Proposition \ref{prop:amenable-main}
are satisfied.\\

\subsubsection{The construction of the subgraph $Z$}

We begin by assigning a weight in the range $\left[0,1\right]$ to
each edge $x\overset{\gamma}{\rightedge}y\in\left[n\right]\times\Gamma\times\left[n\right]$
(regardless of whether or not the edge belongs to $X_{f}$).
\begin{defn}
\label{def:Graph_theoretic_definitions}The \emph{set of supporters}
of an edge $x\overset{\gamma}{\rightedge}y\in\left[n\right]\times\Gamma\times\left[n\right]$
is 
\[
T\left(x\overset{\gamma}{\rightedge}y\right)=\left\{ t\in\Gamma\mid f\left(t\right)f\left(t^{-1}\gamma\right)\left(x\right)=y\right\} 
\]
and the \emph{weight} of this edge is 
\[
w\left(x\overset{\gamma}{\rightedge}y\right)=m\left(T\left(x\overset{\gamma}{\rightedge}y\right)\right)\,\,\text{.}
\]
\end{defn}

Note that $T\left(x\overset{\gamma}{\rightedge}y\right)$ consists
of all elements $t\in\Gamma$ such that the path $x\overset{t^{-1}\gamma}{\rightedge}\overset{t}{\rightedge}$
in the function graph $X_{f}$ ends at $y$. Recall that our eventual
goal is to find a homomorphism $g$ near $f$. One may think of the
weight $w\left(x\overset{\gamma}{\rightedge}y\right)$ as the (normalized)
result of a vote, taken among the elements of $\Gamma$, on whether
the permutation $g\left(\gamma\right)$ should send $x$ to $y$.
Intuitively, if $\delta_{\infty}$ (or $\delta_{1}$) is small then
almost all edges of $X_{f}$ have high weight. Below, we consider
subgraphs of $X_{f}$ that include only the high-weight edges. In
Section \ref{subsec:they-are-groupoids}, we prove that two of these
subgraphs admit an algebraic structure. In Section \ref{subsec:amenable-construct-g}
the algebraic structure gives rise to the sought-after homomorphism
$g$.

For a $\Gamma$-graph $X=\left(V,E\right)$ and a subset $V'\subset V$,
the \emph{induced subgraph} of $X$ on $V'$ is $X'=\left(V',E'\right)$,
where $E'=E\cap\left(V'\times\Gamma\times V'\right)$. We say that
a subgraph $X''$ of $X$ is \emph{induced} if there is a subset $V''\subset V$
such that $X''$ is the induced subgraph of $X$ on $V''$.
{
\begin{defn}
\label{def:from-X_f-to-X_0}~\renewcommand{\labelenumi}{\roman{enumi})}
\begin{enumerate}
\item For $\eps>0$, let $X_{\eps}$ be the subgraph
of $X_{f}$ with vertex set $[n]$  and edge set
\[
E\left(X_{\eps}\right)=\left\{ x\overset{\gamma}{\rightedge}y\in E\left(X_{f}\right)\mid w\left(x\overset{\gamma}{\rightedge}y\right)>1-\eps\right\} \,\,\text{.}
\]
\item For $0<\eps\leq1/6$, let $Y_{\eps}$ be the induced subgraph of $X_{2\eps}$
on the following vertex set:
\[
V\left(Y_{\eps}\right)=\left\{ x\in\left[n\right]\mid\deg_{X_{\eps}}\left(x\right)>2/3\right\} \,\,\text{.}
\]
Explicitly, the edge set $E\left(Y_{\eps}\right)$ consists of all
edges $x\overset{\gamma}{\rightedge}f\left(\gamma\right)\left(x\right)$
for $x\in\left[n\right]$ and $\gamma\in\Gamma$ such that $\deg_{X_{\eps}}\left(x\right)>2/3$,
$\deg_{X_{\eps}}\left(f\left(\gamma\right)\left(x\right)\right)>2/3$
and $w\left(x\overset{\gamma}{\rightedge}y\right)>1-2\eps$.
\item For $0<\eps\leq1/6$, let $Z_{\eps}$ be the induced subgraph of $Y_{\eps}$
on the following vertex set: 
\[
V\left(Z_{\eps}\right)=\left\{ x\in V\left(Y_{\eps}\right)\mid\deg_{Y_{\eps}}\left(x\right)\geq1/2\right\} \,\,\text{.}
\]
\item Finally, set $Z=Z_{1/6}$.
\end{enumerate}
\end{defn}
}
Note that in the definition of $Y_{\eps}$ we use $\deg_{X_{\eps}}$
to filter out the low-degree vertices of $X_{2\eps}$. The interplay
between $\eps$ and $2\eps$ is crucial in our proof in the next section
that $Y_{\eps}$ and $Z_{\eps}$ are well structured. In regard to
the definition of $Z_{\eps}$, in Section \ref{subsec:amenable-construct-g}
we shall see that the degrees in $Y_{\eps}$ are constant within each
connected component, and thus $Z_{\eps}$ is a union of components
of $Y_{\eps}$.
\ \\
\subsubsection{\label{subsec:they-are-groupoids}The $\Gamma$-graphs $Y_{\protect\eps}$
and $Z_{\protect\eps}$ are $\Gamma$-groupoids}
\begin{defn}
\label{def:groupoid}A $\Gamma$-graph $X=\left(V,E\right)$ is a
\emph{$\Gamma$-groupoid} if the following conditions hold:\renewcommand{\labelenumi}{\roman{enumi})}
\begin{enumerate}
\item \textbf{Symmetry:} for every edge $x\overset{\gamma}{\rightedge}y\in E$,
we have $y\overset{\gamma^{-1}}{\rightedge}x\in E$.
\item \textbf{Triangles:} for all $x,y,z\in V$ and $\gamma_{1},\gamma_{2}\in\Gamma$,
if $x\overset{\gamma_{1}}{\rightedge}y\in E$ and $y\overset{\gamma_{2}}{\rightedge}z\in E$
then $x\overset{\gamma_{2}\gamma_{1}}{\rightedge}z\in E$.
\end{enumerate}
\end{defn}

If $h\colon\Gamma\rightarrow\Sym\left(n\right)$ is a homomorphism
then the function graph $X_{h}$ is a $\Gamma$-groupoid, usually
referred to as an \emph{action groupoid. }In Section \ref{subsec:amenable-construct-g},
we investigate general properties of $\Gamma$-groupoids. Here, we
prove that $Y_{\eps}$ and $Z_{\eps}$ are $\Gamma$-groupoids whenever
$\eps\leq1/6$.

Our assumption that $f$ is symmetric comes into play in the next
lemma, which will be used in the sequel without reference.
\pagebreak[2]
\begin{lem}
\label{lem:inverse-edge-weight}Consider an edge $x\overset{\gamma}{\rightedge}y\in\left[n\right]\times\Gamma\times\left[n\right]$.
Then $w\left(x\overset{\gamma}{\rightedge}y\right)=w\left(y\overset{\gamma^{-1}}{\rightedge}x\right)$.
\end{lem}

\begin{proof}
Let $t\in\Gamma$. Then $t\in T\left(x\overset{\gamma}{\rightedge}y\right)$
if and only if $f\left(t\right)f\left(t^{-1}\gamma\right)\left(x\right)=y$
if and only if $\text{\ensuremath{x=f\left(\gamma^{-1}t\right)f\left(\left(\gamma^{-1}t\right)^{-1}\gamma^{-1}\right)\left(y\right)} }$if
and only if $\gamma^{-1}t\in T\left(y\overset{\gamma^{-1}}{\rightedge}x\right)$.
Thus $\gamma^{-1}T\left(x\overset{\gamma}{\rightedge}y\right)=T\left(y\overset{\gamma^{-1}}{\rightedge}x\right)$,
and since $m$ is left invariant we conclude that $w\left(x\overset{\gamma}{\rightedge}y\right)=w\left(y\overset{\gamma^{-1}}{\rightedge}x\right)$.
%Let $t\in\Gamma$. Then 
%\[
%\begin{split}
%t\in T\left(x\overset{\gamma}{\rightedge}y\right) \iff f\left(t\right)f\left(t^{-1}\gamma\right)\left(x\right)=y \iff
%\text{\ensuremath{x=f\left(\gamma^{-1}t\right)f\left(\left(\gamma^{-1}t\right)^{-1}\gamma^{-1}\right)\left(y\right)} } \iff 
%\gamma^{-1}t\in T\left(y\overset{\gamma^{-1}}{\rightedge}x\right).
%\end{split}
%\]
%Thus $\gamma^{-1}T\left(x\overset{\gamma}{\rightedge}y\right)=T\left(y\overset{\gamma^{-1}}{\rightedge}x\right)$,
%and since $m$ is left invariant we conclude that $w\left(x\overset{\gamma}{\rightedge}y\right)=w\left(y\overset{\gamma^{-1}}{\rightedge}x\right)$.
\end{proof}
\begin{lem}
\label{lem:three-strong-edges}Consider a triplet of edges $x\overset{\gamma_{1}}{\rightedge}y$,
$y\overset{\gamma_{2}}{\rightedge}z$ and $x\overset{\gamma_{2}\gamma_{1}}{\rightedge}u$
in $\left[n\right]\times\Gamma\times\left[n\right]$, each of weight
larger than $2/3$. Then $z=u$.
\end{lem}

\begin{proof}
Let $Q=T\left(x\overset{\gamma_{1}}{\rightedge}y\right)\cap\gamma_{2}^{-1}T\left(y\overset{\gamma_{2}}{\rightedge}z\right)\cap\gamma_{2}^{-1}T\left(x\overset{\gamma_{2}\gamma_{1}}{\rightedge}u\right)$.
Each set in the intersection has measure larger than $2/3$, so there
exists an element $t$ in $Q$. Hence,
\begin{align*}
z & =f\left(\gamma_{2}t\right)f\left(\left(\gamma_{2}t\right)^{-1}\gamma_{2}\right)\left(y\right) & \gamma_{2}t\in T\left(y\overset{\gamma_{2}}{\rightedge}z\right)\\
 & =f\left(\gamma_{2}t\right)f\left(t^{-1}\right)\left(y\right)\\
 & =f\left(\gamma_{2}t\right)f\left(t^{-1}\right)f\left(t\right)f\left(t^{-1}\gamma_{1}\right)\left(x\right) & t\in T\left(x\overset{\gamma_{1}}{\rightedge}y\right)\\
 & =f\left(\gamma_{2}t\right)f\left(t^{-1}\gamma_{1}\right)\left(x\right) & \text{\ensuremath{f} is symmetric}\\
 & =f\left(\gamma_{2}t\right)f\left(\left(\gamma_{2}t\right)^{-1}\gamma_{2}\gamma_{1}\right)\left(x\right)\\
 & =u\,\,\text{.} & \gamma_{2}t\in T\left(x\overset{\gamma_{2}\gamma_{1}}{\rightedge}u\right)
\end{align*}

\end{proof}
\begin{lem}
\label{lem:complete-the-triangle}Consider $\eps_{1},\eps_{2}>0$
and a pair of edges $x\overset{\gamma_{1}}{\rightedge}y$ and $y\overset{\gamma_{2}}{\rightedge}z$
in $\left[n\right]\times\Gamma\times\left[n\right]$, such that $w\left(x\overset{\gamma_{1}}{\rightedge}y\right)>1-\eps_{1}$
and $w\left(y\overset{\gamma_{2}}{\rightedge}z\right)>1-\eps_{2}$.
Then $w\left(x\overset{\gamma_{2}\gamma_{1}}{\rightedge}z\right)>1-\eps_{1}-\eps_{2}$.
\end{lem}

\begin{proof}
Let $Q=\gamma_{2}T\left(x\overset{\gamma_{1}}{\rightedge}y\right)\cap T\left(y\overset{\gamma_{2}}{\rightedge}z\right)$.
Then $m\left(Q\right)>1-\eps_{1}-\eps_{2}$, and thus it suffices
to show that $Q\subset T\left(x\overset{\gamma_{2}\gamma_{1}}{\rightedge}z\right)$.
Indeed, if $t\in Q$ then $t\in T\left(x\overset{\gamma_{2}\gamma_{1}}{\rightedge}z\right)$
because 
\begin{align*}
f\left(t\right)f\left(t^{-1}\gamma_{2}\gamma_{1}\right)\left(x\right) & =f\left(t\right)f\left(t^{-1}\gamma_{2}\right)f\left(\gamma_{2}^{-1}t\right)f\left(t^{-1}\gamma_{2}\gamma_{1}\right)\left(x\right) & \text{\ensuremath{f} is symmetric}\\
 & =f\left(t\right)f\left(t^{-1}\gamma_{2}\right)\left(y\right) & \gamma_{2}^{-1}t\in T\left(x\overset{\gamma_{1}}{\rightedge}y\right)\\
 & =z\,\,\text{.} & t\in T\left(y\overset{\gamma_{2}}{\rightedge}z\right)
\end{align*}
\end{proof}
The above lemma shows that the composition of two high-weight edges
results in a high-weight edge, but with some decrease in weight. This
deterioration makes it difficult to grow a large $\Gamma$-groupoid
edge by edge inside $X_{2\eps}$ . This difficulty is addressed by
the following lemma, which motivates the distinction between $\eps$
and $2\eps$ in the definition of $Y_{\eps}$.

\begin{lem}
\label{lem:good-compositions}Let $0<\eps\leq1/6$ and consider a
pair of edges $x\overset{\gamma_{1}}{\rightedge}y$ and $y\overset{\gamma_{2}}{\rightedge}z$
in $\left[n\right]\times\Gamma\times\left[n\right]$, such that

\begin{align}
\deg_{X_{\eps}}\left(x\right),\deg_{X_{\eps}}\left(y\right),\deg_{X_{\eps}}\left(z\right) & >2/3\,\,\text{,}\nonumber \\
w\left(x\overset{\gamma_{1}}{\rightedge}y\right) & >1-2\eps\,\,\text{and}\label{eq:good-compositions-1}\\
w\left(y\overset{\gamma_{2}}{\rightedge}z\right) & >1-2\eps\,\,\text{.}\label{eq:good-compositions-2}
\end{align}
Then $w\left(x\overset{\gamma_{2}\gamma_{1}}{\rightedge}z\right)>1-2\eps$.
\end{lem}

\begin{proof}
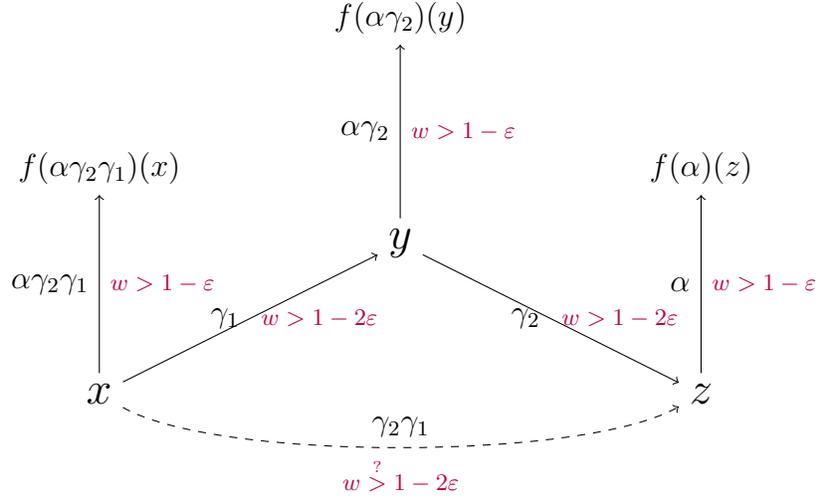
\begin{figure}
	\begin{centering}
		\begin{tikzpicture}[scale=2]
		\node[draw=none, color=black] (x) at (-2,0) {\Large ${x}$};
		\node[draw=none, color=black] (z) at (2,0) {\Large ${z}$};
		\node[draw=none, color=black] (y) at (0,1) {\Large ${y}$};
		\node[draw=none, color=black] (agg) at (-2,1.5) { $f(\alpha\gamma_2\gamma_1)(x) $};
		\node[draw=none, color=black] (ag) at (0,2.5) { $f(\alpha\gamma_2)(y) $};
		\node[draw=none, color=black] (a) at (2,1.5) { $f(\alpha)(z) $};

		\draw[black,->, solid] (x) to  %[out=135,in=45,looseness=6]
		node[midway,left]{$\gamma_1$}node[midway,right]{\scriptsize \color{purple}$w>1-2\varepsilon$}  (y)  ; 
		\draw[black,->, solid] (y) to  %[out=135,in=45,looseness=6]
		node[midway,left]{$\gamma_2$} node[midway,right]{\scriptsize \color{purple} $w>1-2\varepsilon$}  (z); 
		\draw[black,->, dashed] (x) to  [out=-30,in=210, looseness=0.5]
		node[midway,above]{\color{black} $\gamma_2\gamma_1$} node[midway,below]{\scriptsize \color{purple}$w\overset{?}{>}1-2\varepsilon$} (z)  ; 
		
		\draw[black,->, solid] (x) to  %[out=135,in=45,looseness=6]
		node[midway,left]{$\alpha\gamma_2\gamma_1$}node[midway,right]{\scriptsize \color{purple}$w>1-\varepsilon$}  (agg)  ; 
		\draw[black,->, solid] (y) to  %[out=135,in=45,looseness=6]
		node[midway,left]{$\alpha\gamma_2$}node[midway,right]{\scriptsize \color{purple}$w>1-\varepsilon$}  (ag)  ; 
		\draw[black,->, solid] (z) to  %[out=135,in=45,looseness=6]
		node[midway,left]{$\alpha$}node[midway,right]{\scriptsize \color{purple}$w>1-\varepsilon$}  (a)  ;

		\end{tikzpicture}	
	\end{centering}
	
	\caption{
		The high degrees of $x$, $y$ and $z$ guarantee the existence of
		$\alpha\in\Gamma$ such that the three vertical edges have weight larger
		than $1-\eps$. Then, Lemma \ref{lem:three-strong-edges} ensures that
		the three vertices on top are the same vertex $u$.
		Finally, Lemma \ref{lem:complete-the-triangle}, applied to the path
		$x\overset{\alpha\gamma_{2}\gamma_{1}}{\rightedge}u
		\overset{\alpha^{-1}}{\rightedge}z$,
		implies that the bottom edge has weight larger than $1-2\eps$.
	}
	\label{fig:pigeonhole-groupoid}
\end{figure}
The proof is illustrated in Figure \ref{fig:pigeonhole-groupoid}.
Let $Q=\left({\rm \OE}_{X_{\eps}}\left(x\right)\gamma_{1}^{-1}\gamma_{2}^{-1}\right)\cap\left(\OE_{X_{\eps}}\left(y\right)\gamma_{2}^{-1}\right)\cap{\rm \OE}_{X_{\eps}}\left(z\right)$.
Each set in the intersection has measure larger than $2/3$, so there
exists an element $\alpha$ in $Q$. Then,
\begin{align}
w\left(x\overset{\alpha\gamma_{2}\gamma_{1}}{\rightedge}f\left(\alpha\gamma_{2}\gamma_{1}\right)\left(x\right)\right) & >1-\eps\,\,\text{,}\label{eq:nondeterioration-1}\\
w\left(y\overset{\alpha\gamma_{2}}{\rightedge}f\left(\alpha\gamma_{2}\right)\left(y\right)\right) & >1-\eps\,\,\text{and}\label{eq:non-deterioration-2}\\
w\left(z\overset{\alpha}{\rightedge}f\left(\alpha\right)\left(z\right)\right) & >1-\eps\,\,\text{.}\label{eq:non-deterioraton-3}
\end{align}
Lemma \ref{lem:three-strong-edges} applies to the triplet of edges
$x\overset{\gamma_{1}}{\rightedge}y$, $y\overset{\alpha\gamma_{2}}{\rightedge}f\left(\alpha\gamma_{2}\right)\left(y\right)$
and $x\overset{\alpha\gamma_{2}\gamma_{1}}{\rightedge}f\left(\alpha\gamma_{2}\gamma_{1}\right)\left(x\right)$
since $\eps\leq1/6$ and by virtue of (\ref{eq:good-compositions-1}),
(\ref{eq:non-deterioration-2}) and (\ref{eq:nondeterioration-1}),
and thus $f\left(\alpha\gamma_{2}\right)\left(y\right)=f\left(\alpha\gamma_{2}\gamma_{1}\right)\left(x\right)$.
Similarly, the lemma applies to the triplet of edges $y\overset{\gamma_{2}}{\rightedge}z$,
$z\overset{\alpha}{\rightedge}f\left(\alpha\right)\left(z\right)$
and $y\overset{\alpha\gamma_{2}}{\rightedge}f\left(\alpha\gamma_{2}\right)\left(y\right)$
due to (\ref{eq:good-compositions-2}), (\ref{eq:non-deterioraton-3})
and (\ref{eq:non-deterioration-2}), and thus $f\left(\alpha\right)\left(z\right)=f\left(\alpha\gamma_{2}\right)\left(y\right)$.
Hence 
\begin{equation}
f\left(\alpha\gamma_{2}\gamma_{1}\right)\left(x\right)=f\left(\alpha\right)\left(z\right)\,\,\text{.}\label{eq:good-composition-same-vertex}
\end{equation}

By (\ref{eq:good-composition-same-vertex}), Lemma \ref{lem:complete-the-triangle}
applies to the pair of edges $x\overset{\alpha\gamma_{2}\gamma_{1}}{\rightedge}f\left(\alpha\gamma_{2}\gamma_{1}\right)\left(x\right)$
and $f\left(\alpha\right)\left(z\right)\overset{\alpha^{-1}}{\rightedge}z$,
and thus
\[
w\left(x\overset{\gamma_{2}\gamma_{1}}{\rightedge}z\right)>1-2\eps
\]
by (\ref{eq:nondeterioration-1}) and (\ref{eq:non-deterioraton-3}).
\end{proof}
\begin{prop}
\label{prop:they-are-groupoids}Let $0<\eps\leq1/6$. Then $Y_{\eps}$
and $Z_{\eps}$ are $\Gamma$-groupoids.
\end{prop}

\begin{proof}
The function graph $X_{f}$ is a $\Gamma$-graph, so the same is true
for its subgraph $X_{2\eps}$. By Lemma \ref{lem:inverse-edge-weight},
$X_{2\eps}$ satisfies the symmetry condition of Definition \ref{def:groupoid},
and thus so does its induced subgraph $Y_{\eps}$. Furthermore, since
$\eps\leq1/6$, Lemma \ref{lem:good-compositions} implies that $Y_{\eps}$
satisfies the triangles condition, and is thus a $\Gamma$-groupoid.
Therefore, the induced subgraph $Z_{\eps}$ of $Y_{\eps}$ is a $\Gamma$-groupoid
as well.
\end{proof}
\ \\

\subsubsection{\label{subsec:amenable-construct-g}Construction of the homomorphism
$g$}

Recall that we want to construct a finite set $V_{1}$, of cardinality
not much larger than $n$, and a homomorphism $g\colon\Gamma\rightarrow\Sym\left(V_{1}\right)$
such that the subgraph $Z=Z_{1/6}$ of $X_{f}$ embeds into the action groupoid
 $X_{g}$. First, we investigate further generalities on $\Gamma$-groupoids.

Let $X=\left(V,E\right)$ be a $\Gamma$-groupoid. Two vertices $x,y\in V$
are \emph{connected} if there is $\gamma\in\Gamma$ such that $x\overset{\gamma}{\rightedge}y\in E$.
Connectedness induces an equivalence relation on $V$. The induced
subgraphs of $X$ on the equivalence classes are the \emph{connected
components}, or \emph{components} for short, of the $\Gamma$-groupoid
$X$. If there is just one component, we say that $X$ is \emph{connected}.
Each component of $X$ is a connected $\Gamma$-groupoid. For $x\in V$,
write $C_{X,x}$, or $C_{x}$, for the component of $x$ in $X$.

The \emph{stabilizer} of $x\in V$ is $\Gamma_{x}=\left\{ \gamma\in\Gamma\mid x\overset{\gamma}{\rightedge}x\in E\right\} $.
Note that $\Gamma_{x}$ is a subgroup of $\Gamma$. Also, if $x\overset{\gamma}{\rightedge}y\in E$,
then $\Gamma_{y}=\gamma\Gamma_{x}\gamma^{-1}$ and $\OE\left(y\right)\gamma=\OE\left(x\right)$.
Hence, if the $\Gamma$-groupoid $X$ is connected, then the index
$\left[\Gamma\colon\Gamma_{x}\right]$ is the same for all $x\in V$,
and the same is true for the degree $\deg\left(x\right)$. These numbers
are, respectively, the the \emph{index} $\index\left(X\right)$ and
\emph{degree} $\deg\left(X\right)$ of the connected groupoid $X$.

By Proposition \ref{prop:they-are-groupoids}, if $0<\eps\leq1/6$
then $Y_{\eps}$ is a $\Gamma$-groupoid. Recalling Definition \ref{def:from-X_f-to-X_0}(iii),
we see that in this case $Z_{\eps}$ is the union of the components
of $Y_{\eps}$ that have degree at least $1/2$.

The proof of the following proposition is straightforward and is left
to the reader.
\begin{prop}
\label{prop:Groupoid_embedding_action}Let $X=\left(V,E\right)$ be
a connected $\Gamma$-groupoid and $x\in V$. Let $h\colon\Gamma\rightarrow\Sym\left(\Gamma/\Gamma_{x}\right)$
be the action of $\Gamma$ on $\Gamma/\Gamma_{x}$ by left multiplication.
Then, the function $\varphi\colon V\rightarrow\Gamma/\Gamma_{x}$
defined by

\[
\varphi\left(y\right)=\gamma\Gamma_{x}\qquad\forall x\overset{\gamma}{\rightedge}y\in E
\]
is well defined and injective. Furthermore, it is an embedding of
the $\Gamma$-graph $X$ into the action groupoid $X_{h}$.
\end{prop}

Let $X=\left(V,E\right)$ be a finite (i.e. $\left|V\right|<\infty$)
connected groupoid. For $x\in V$, Proposition \ref{prop:Groupoid_embedding_action}
embeds $X$, which has $\left|V\right|$ vertices, in an action groupoid
that has $\left[\Gamma\colon\Gamma_{x}\right]$ vertices. We investigate
the ratio $\frac{\left|V\right|}{\left[\Gamma\colon\Gamma_{x}\right]}$.
For $x\in V$ and $\gamma\in\OE\left(x\right)$, write $\gamma\cdot x$
for the unique element $y$ of $V$ such that $x\overset{\gamma}{\rightedge}y\in E$.
For $x\in V$, $\OE\left(x\right)$ is a union of left cosets of $\Gamma_{x}$,
and for $\gamma_{1},\gamma_{2}\in\OE\left(x\right)$ we have $\gamma_{1}\cdot x=\gamma_{2}\cdot x$
if and only if $\gamma_{1}\Gamma_{x}=\gamma_{2}\Gamma_{x}$. Since
$X$ is connected, this means that $\OE\left(x\right)$ is a union
of $\left|V\right|$ distinct left cosets of $\Gamma_{x}$. Hence
$\left|V\right|m\left(\Gamma_{x}\right)=m\left(\OE\left(x\right)\right)=\deg\left(X\right)$.
But $\left[\Gamma\colon\Gamma_{x}\right]=\frac{1}{m\left(\Gamma_{x}\right)}$,
and thus if $\deg\left(X\right)>0$ then $\left[\Gamma\colon\Gamma_{x}\right]<\infty$
and 
\begin{equation}
\frac{\left|V\right|}{\left[\Gamma\colon\Gamma_{x}\right]}=\deg\left(X\right)\,\,\text{.}\label{eq:index-degree}
\end{equation}

It is possible to apply Proposition \ref{prop:Groupoid_embedding_action}
to each component of $Y_{1/6}$ in order to embed $Y_{1/6}$ into
an action groupoid. However, in order to obtain sufficiently good
bounds on the number of vertices in the action groupoid, we do the same
to $Z=Z_{1/6}$ rather than $Y_{1/6}$.

The definition of the set $V_{1}$, the homomorphism $g\colon\Gamma\rightarrow\Sym\left(V_{1}\right)$
and the embedding of $Z$ into the action groupoid $X_g$
proceeds as follows.
Let $\left\{ C_{i}\right\} _{i=1}^{m}$ be the components of $Z$.
Fix a vertex $x_{i}\in V\left(C_{i}\right)$ for each $1\le i\leq m$,
and write $g_{i}\colon\Gamma\to\Sym\left(\Gamma/\Gamma_{x_{i}}\right)$
for the action of $\Gamma$ on $\Gamma/\Gamma_{x_{i}}$ by left multiplication.
By Proposition \ref{prop:Groupoid_embedding_action}, each $C_{i}$ embeds into $X_{g_{i}}$,
and thus $Z$ embeds into $\coprod_{i=1}^{m}X_{g_{i}}$. More precisely,
write $V_{1}=\coprod_{i=1}^{m}\Gamma/\Gamma_{x_{i}}$ and let $g\colon\Gamma\to\Sym\left(V_{1}\right)$
be the action by left multiplication. Each $C_{i}$ embeds into $X_{g_{i}}$
by an embedding $\varphi_{i}:V\left(C_{i}\right)\to\Gamma/\Gamma_{x_{i}}$.
These embeddings give rise to an embedding $\varphi:V\left(Z\right)\to V_{1}$
of $Z$ into $X_{g}$.

We evaluate the cardinality of $V_{1}$ as follows.
\begin{align}
\left|V_{1}\right| & =\sum_{i=1}^{c}\left[\Gamma:\Gamma_{x}\right]\nonumber \\
 & =\sum_{i=1}^{c}\left|V\left(C_{i}\right)\right|\cdot\left(\deg\left(C_{i}\right)\right)^{-1} & \text{by (\ref{eq:index-degree})}\nonumber \\
 & =\sum_{x\in V\left(Z\right)}\left(\deg_{Z}x\right)^{-1}\,\,\text{.}\label{eq:V1-sum-of-inverses}
\end{align}
We shall see in Lemma \ref{lem:Z-graph-analysis}(iv) that (\ref{eq:V1-sum-of-inverses})
is bounded from above by $\left(1+O\left(\delta_{1}\right)\right)\left|V\left(Z\right)\right|$.\\

\subsubsection{Proof of Proposition \ref{prop:amenable-main}(ii)-(iv)}

We constructed the subgraph $Z$ of $X_{f}$, the homomorphism $g\colon\Gamma\rightarrow\Sym\left(V_{1}\right)$
and an embedding $\varphi\colon V\left(Z\right)\rightarrow V_{1}$
of $Z$ into $X_{g}$. To complete the proof of Proposition \ref{prop:amenable-main},
it remains to prove the lower bounds on $\left|V\left(Z\right)\right|$
and $\left|D_{Z}\left(\gamma\right)\right|$ and the upper bound on
$\left|V_{1}\right|$. To do so, we study quantitative properties
of $X_{\eps}$, $Y_{\eps}$ and  $Z_{\eps}$.

In the following three lemmas, we justify some of the steps by invoking
Markov's inequality. By this we are referring to the fact that if
$a_{1},\dotsc,a_{m}$ are real numbers in the interval $\left[0,1\right]$
and $\theta>0$, then 
\begin{equation}\label{eq:markov-ineq}
\left|\left\{ i\mid a_{i}>1-\theta\right\} \right|\geq\left(1-\frac{1}{\theta}\left(1-\frac{1}{m}\sum_{i=1}^{m}a_{i}\right)\right)m\,\,\text{.}
\end{equation}
\begin{lem}
\label{lem:X-graph-analysis}Let $\eps>0$. Then\renewcommand{\labelenumi}{\roman{enumi})}
\begin{enumerate}
\item $\left|D_{X_{\eps}}\left(\gamma\right)\right|\geq\left(1-\frac{\delta_{\infty}}{\eps}\right)n$
for every $\gamma\in\Gamma$.
\item $\int\left|D_{X_{\eps}}\left(\gamma\right)\right|\dm\left(\gamma\right)\geq\left(1-\frac{\delta_{1}}{\eps}\right)n$.
\end{enumerate}
\end{lem}

\begin{proof}
For $\gamma,t\in\Gamma$ and $x\in\left[n\right]$, write ${\bf 1}_{\gamma,t,x}=\begin{cases}
1 & f\left(\gamma\right)\left(x\right)=f\left(t\right)f\left(t^{-1}\gamma\right)\left(x\right)\\
0 & {\rm otherwise}
\end{cases}$. Then
\begin{equation}
\frac{1}{n}\sum_{x=1}^{n}{\bf 1}_{\gamma,t,x}=1-d^{H}\left(f\left(t^{-1}\right)f\left(\gamma\right),f\left(t^{-1}\gamma\right)\right)\geq1-\delta_{\infty}\qquad\forall\gamma\in\Gamma\forall t\in\Gamma\label{eq:strong-group-elements-1}
\end{equation}
and 
\begin{equation}
\int{\bf 1}_{\gamma,t,x}\dm\left(t\right)=w\left(x\overset{\gamma}{\rightedge}f\left(\gamma\right)\left(x\right)\right)\qquad\forall\gamma\in\Gamma\forall x\in\left[n\right]\,\,\text{.}\label{eq:strong-group-elements-2}
\end{equation}
Hence, for $\gamma\in\Gamma$,
\[
\frac{1}{n}\sum_{x\in\left[n\right]}w\left(x\overset{\gamma}{\rightedge}f\left(\gamma\right)\left(x\right)\right)=\int\left(\frac{1}{n}\sum_{x\in\left[n\right]}{\bf 1}_{\gamma,t,x}\right)\dm\left(t\right)\geq1-\delta_{\infty}\,\,\text{,}
\]
and thus (i) follows from Markov's inequality (\ref{eq:markov-ineq}): 
\[
\left|D_{X_{\eps}}\left(\gamma\right)\right|=\left|\left\{ x\in\left[n\right]\mid w\left(x\overset{\gamma}{\rightedge}f\left(\gamma\right)\left(x\right)\right)>1-\epsilon\right\} \right|\geq\left(1-\frac{\delta_{\infty}}{\epsilon}\right)n\,\,\text{.}
\]
Now, using (\ref{eq:strong-group-elements-2}), (\ref{eq:strong-group-elements-1})
and the inverse-invariance of $m$,
\begin{align}
\frac{1}{n}\sum_{x\in\left[n\right]}\int w\left(x\overset{\gamma}{\rightedge}f\left(\gamma\right)\left(x\right)\right)\dm\left(\gamma\right) & =\int\int\left(\frac{1}{n}\sum_{x\in\left[n\right]}{\bf 1}_{\gamma,t,x}\right)\dm\left(t\right)\dm\left(\gamma\right)\nonumber \\
 & =\int\int\left(1-d^{H}\left(f\left(t^{-1}\right)f\left(\gamma\right),f\left(t^{-1}\gamma\right)\right)\right)\dm\left(t\right)\dm\left(\gamma\right)\nonumber \\
 & =1-\delta_{1}\,\,\text{.}\label{eq:X-graph-analysis}
\end{align}
Thus (ii) follows:
\begin{align*}
\int\left|D_{X_{\eps}}\left(\gamma\right)\right|\dm\left(\gamma\right) & =\int\left|\left\{ x\in\left[n\right]\mid w\left(x\overset{\gamma}{\rightedge}f\left(\gamma\right)\left(x\right)\right)>1-\eps\right\} \right|\dm\left(\gamma\right)\\
 & \geq\int\left(1-\frac{1}{\eps}\left(1-\frac{1}{n}\sum_{x\in\left[n\right]}w\left(x\overset{\gamma}{\rightedge}f\left(\gamma\right)\left(x\right)\right)\right)\right)n\dm\left(\gamma\right) & \text{by Markov (\ref{eq:markov-ineq})}\\
 & =\left(1-\frac{\delta_{1}}{\eps}\right)n\,\,\text{.} & \text{by (\ref{eq:X-graph-analysis})}
\end{align*}
\end{proof}
\begin{lem}
\label{lem:Y-graph-analysis}Let $0<\eps\leq\frac{1}{6}$. Then\renewcommand{\labelenumi}{\roman{enumi})}
\begin{enumerate}
\item $|V\left(Y_{\eps}\right)|\geq\left(1-\frac{3\delta_{1}}{\eps}\right)n$.
\item $\left|D_{Y_{\eps}}\left(\gamma\right)\right|\geq\left(1-\frac{6.5\delta_{\infty}}{\eps}\right)n$
for every $\gamma\in\Gamma$.
\item $\int\left|D_{Y_{\eps}}\left(\gamma\right)\right|\dm\left(\gamma\right)\geq\left(1-\frac{6.5\delta_{1}}{\eps}\right)n$.
\end{enumerate}
\end{lem}

\begin{proof}
To prove (i), we compute
\begin{align*}
\left|V\left(Y_{\eps}\right)\right| & =\left|\left\{ x\in\left[n\right]\mid\deg_{X_{\eps}}\left(x\right)>1-1/3\right\} \right|\\
 & \geq\left(1-3\left(1-\frac{1}{n}\sum_{x\in\left[n\right]}\deg_{X_{\eps}}\left(x\right)\right)\right)n & \text{by Markov (\ref{eq:markov-ineq})}\\
 & =\left(1-3\left(1-\frac{1}{n}\int\left|D_{X_{\eps}}\left(\gamma\right)\right|\dm\left(\gamma\right)\right)\right)n & \text{by Fact \ref{fact:deg-and-domain}}\\
 & \geq\left(1-\frac{3\delta_{1}}{\eps}\right)n\,\,\text{.} & \text{by Lemma \ref{lem:X-graph-analysis}(ii)}
\end{align*}
Now, for $\gamma\in\Gamma$,
\begin{align}
\left|D_{Y_{\eps}}\left(\gamma\right)\right| & =\left|D_{X_{2\eps}}\left(\gamma\right)\cap V\left(Y_{\eps}\right)\cap f\left(\gamma\right){}^{-1}V\left(Y_{\eps}\right)\right|\nonumber \\
 & =\left|D_{X_{2\eps}}\left(\gamma\right)\setminus\left(\left(\left[n\right]\setminus V\left(Y_{\eps}\right)\right)\cup\left(\left[n\right]\setminus f\left(\gamma^{-1}\right)V\left(Y_{\eps}\right)\right)\right)\right|\nonumber \\
 & \geq\left|D_{X_{2\eps}}\left(\gamma\right)\right|-2\left(n-|V\left(Y_{\eps}\right)|\right)\nonumber \\
 & \geq\left|D_{X_{2\eps}}\left(\gamma\right)\right|-\frac{6\delta_{1}}{\eps}n\,\,\text{.} & \text{by (i)}\label{eq:Y-graph-large-domain}
\end{align}
Then (ii) follows from (\ref{eq:Y-graph-large-domain}) and Lemma
\ref{lem:X-graph-analysis}(i), and (iii) follows by integrating (\ref{eq:Y-graph-large-domain})
and using Lemma \ref{lem:X-graph-analysis}(ii).
\end{proof}
Proposition \ref{prop:amenable-main}(ii)-(iii) follows from Lemma
\ref{lem:Z-graph-analysis}(i)-(iii) by plugging in $\eps=1/6$. In
light of (\ref{eq:V1-sum-of-inverses}), Proposition \ref{prop:amenable-main}(iv)
follows from Lemma \ref{lem:Z-graph-analysis}(iv).
\begin{lem}
\label{lem:Z-graph-analysis}Let $0<\eps\leq\frac{1}{6}$ and assume
that $\delta_{1}\leq\eps/13$. Then\renewcommand{\labelenumi}{\roman{enumi})}
\begin{enumerate}
\item $\left|Z_{\eps}\right|\geq\left(1-\frac{16\delta_{1}}{\eps}\right)n$.
\item $\left|D_{Z_{\eps}}\left(\gamma\right)\right|\geq\left(1-\frac{19.5\delta_{\infty}}{\eps}\right)n$
for every $\gamma\in\Gamma$.
\item $\int\left|D_{Z_{\eps}}\left(\gamma\right)\right|\dm\left(\gamma\right)\geq\left(1-\frac{19.5\delta_{1}}{\eps}\right)n$.
\item $\sum_{x\in V\left(Z_{\eps}\right)}\left(\deg_{Z_{\eps}}\left(x\right)\right)^{-1}\leq\left(1+\frac{13\delta_{1}}{\eps}\right)n$.
\end{enumerate}
\end{lem}

\begin{proof}
We first bound the average degree in $Y_{\eps}$.
\begin{align}
\frac{1}{\left|V\left(Y_{\eps}\right)\right|}\sum_{x\in V\left(Y_{\eps}\right)}\deg_{Y_{\eps}}\left(x\right) & \geq\frac{1}{n}\sum_{x\in V\left(Y_{\eps}\right)}\deg_{Y_{\eps}}\left(x\right)\nonumber \\
 & =\frac{1}{n}\int\left|D_{Y_{\eps}}\left(\gamma\right)\right|\dm\left(\gamma\right) & \text{by Fact \ref{fact:deg-and-domain}}\nonumber \\
 & \geq1-\frac{6.5\delta_{1}}{\eps} & \text{by Lemma \ref{lem:Y-graph-analysis}(iii)}\label{eq:Z-graph-analysis-1}\\
 & \geq1/2\,\,\text{.} & \delta_{1}\le\eps/13\label{eq:Z-graph-analysis-2}
\end{align}
Recall that $Z_{\eps}$ is attained from $Y_{\eps}$ by removing the
components of degree smaller than $1/2$, and hence, by (\ref{eq:Z-graph-analysis-2}),
the average degree in $Z_{\eps}$ is larger or equal to that of $Y_{\eps}$.
Thus, by (\ref{eq:Z-graph-analysis-1}), 
\begin{equation}
\frac{1}{\left|V\left(Z_{\eps}\right)\right|}\sum_{x\in V\left(Z_{\eps}\right)}\deg_{Z_{\eps}}\left(x\right)\geq1-\frac{6.5\delta_{1}}{\eps}\,\,\text{.}\label{eq:z-graph-analysis-5}
\end{equation}
To prove (i), we compute
\begin{align}
\left|V\left(Z_{\eps}\right)\right| & =\left|\left\{ x\in V\left(Y_{\eps}\right)\mid\deg_{Y_{\eps}}\left(x\right)\geq\frac{1}{2}\right\} \right|\nonumber \\
 & \geq\left(1-2\left(1-\frac{1}{\left|V\left(Y_{\eps}\right)\right|}\sum_{x\in V\left(Y_{\eps}\right)}\deg_{Y_{\eps}}\left(x\right)\right)\right)\left|V\left(Y_{\eps}\right)\right| & \text{by Markov (\ref{eq:markov-ineq})}\nonumber \\
 & \geq\left(1-\frac{13\delta_{1}}{\eps}\right)\left|V\left(Y_{\eps}\right)\right| & \text{by (\ref{eq:Z-graph-analysis-1})}\label{eq:z-graph-analysis-3}\\
 & \geq\left(1-\frac{16\delta_{1}}{\eps}\right)n\,\,\text{.} & \text{by Lemma \ref{lem:Y-graph-analysis}(i)}\nonumber 
\end{align}
Now, for $\gamma\in\Gamma$,
\begin{align}
\left|D_{Z_{\eps}}\left(\gamma\right)\right| & \geq\left|D_{Y_{\eps}}\left(\gamma\right)\right|-\left(\left|V\left(Y_{\eps}\right)\right|-\left|V\left(Z_{\eps}\right)\right|\right)\nonumber \\
 & \geq\left|D_{Y_{\eps}}\left(\gamma\right)\right|-\frac{13\delta_{1}}{\eps}n\,\,\text{.} & \text{by (\ref{eq:z-graph-analysis-3})}\label{eq:z-graph-analysis-4}
\end{align}
Then (ii) follows from (\ref{eq:z-graph-analysis-4}) and Lemma \ref{lem:Y-graph-analysis}(ii),
and (iii) follows from integrating (\ref{eq:z-graph-analysis-4})
and using Lemma \ref{lem:Y-graph-analysis}(iii).

Finally, we prove (iv). For $x\in V\left(Z_{\eps}\right)$ we have
$\deg_{Z_{\eps}}\left(x\right)\geq\frac{1}{2}$ and thus $\left(\deg_{Z_{\eps}}\left(x\right)\right)^{-1}\leq1+2\left(1-\deg_{Z_{\eps}}\left(x\right)\right)$.
Hence,
\begin{align*}
\sum_{x\in V\left(Z_{\eps}\right)}\left(\deg_{Z_{\eps}}\left(x\right)\right)^{-1} & \leq\sum_{x\in V\left(Z_{\eps}\right)}\left(1+2\left(1-\deg_{Z_{\eps}}\left(x\right)\right)\right)\\
 & =|V\left(Z_{\eps}\right)|\left(1+\frac{2}{|V\left(Z_{\eps}\right)|}\sum_{x\in V\left(Z_{\eps}\right)}\left(1-\deg_{Z_{\eps}}\left(x\right)\right)\right)\\
 & \leq n\left(1+\frac{13\delta_{1}}{\eps}\right)\,\,\text{.} & \text{by (\ref{eq:z-graph-analysis-5})}
\end{align*}
\end{proof}
%\ \\
\subsubsection{From almost vanishing on a coamenable subgroup to a nearby homomorphism}

The following theorem strengthens Theorem \ref{thm:intro-amenable-positive}
(up to the constants). We will use it in Section \ref{sec:special-linear}
in the special case where $\left[\Gamma\colon\Delta\right]<\infty$.
For $\delta>0$ and permutations $\sigma$ and $\tau$, the proof
uses the following notation 
\begin{equation}\label{approx_notation} 
\sigma\approx_{\delta}\tau\quad \mathrm{means}\quad d^{H}\left(\sigma,\tau\right)\leq\delta\,\,\text{.}
\end{equation}

\begin{thm}
\label{thm:almost-trivial-on-finite-index}Let $\Gamma$ be a group,
$f\colon\Gamma\rightarrow\Sym\left(n\right)$ a function and $\Delta\lhd\Gamma$
a normal subgroup such that $\Gamma/\Delta$ is a discrete amenable
group. Write $\delta_{\infty}=\defect_{\infty}\left(f\right)$ and
$\delta_{\Delta}=\sup\left\{ d^{H}\left(f\left(\gamma\right),\id\right)\mid\gamma\in\Delta\right\} $.
Then, there is $n\leq N\leq\left(1+2436\delta_{\infty}+1218\delta_{\Delta}\right)n$
and a homomorphism $h\colon\Gamma\rightarrow\Sym\left(N\right)$ such
that $d_{\infty}\left(h,f\right)\leq4079\delta_{\infty}+2040\delta_{\Delta}$.
\end{thm}

\begin{proof}
Let $T$ be a transversal for the set of left cosets of $\Delta$
in $\Gamma$. Define a function $\bar{f}\colon\Gamma/\Delta\rightarrow\Sym\left(n\right)$
by letting $\bar{f}\left(\gamma\Delta\right)=f\left(\gamma\right)$
for $\gamma\in T$. Let $\gamma_{1},\gamma_{2}\in T$ and take $\gamma\in T$
and $\alpha\in\Delta$ such that $\gamma=\gamma_{1}\gamma_{2}\alpha$.
Then
\[
\bar{f}\left(\left(\gamma_{1}\Delta\right)\left(\gamma_{2}\Delta\right)\right)=f\left(\gamma\right)=f\left(\gamma_{1}\gamma_{2}\alpha\right)\approx_{2\delta_{\infty}}f\left(\gamma_{1}\right)f\left(\gamma_{2}\right)f\left(\alpha\right)\approx_{\delta_{\Delta}}f\left(\gamma_{1}\right)f\left(\gamma_{2}\right)=\bar{f}\left(\gamma_{1}\Delta\right)\bar{f}\left(\gamma_{2}\Delta\right)\,\,\text{.}
\]
Hence, $\defect_{\infty}\left(\bar{f}\right)\leq2\delta_{\infty}+\delta_{\Delta}$.
Therefore, by Theorem \ref{thm:intro-amenable-positive} applied to
$\bar{f}$, there is $n\leq N\leq\left(1+1218\left(2\delta_{\infty}+\delta_{\Delta}\right)\right)n$
and a homomorphism $\bar{h}\colon\Gamma/\Delta\rightarrow\Sym\left(N\right)$
such that $d_{\infty}\left(\bar{h},\bar{f}\right)\leq2039\left(2\delta_{\infty}+\delta_{\Delta}\right)$.
Define $h\colon\Gamma\rightarrow\Sym\left(N\right)$ by letting $h\left(\gamma\right)=\bar{h}\left(\gamma\Delta\right)$
for $\gamma\in\Gamma$. Consider $\gamma\in\Gamma$ and take $\alpha\in\Delta$
such that $\gamma\alpha\in T$. Then
\[
h\left(\gamma\right)=\bar{h}\left(\gamma\Delta\right)\approx_{4078\delta_{\infty}+2039\delta_{\Delta}}\bar{f}\left(\gamma\Delta\right)=f\left(\gamma\alpha\right)\approx_{\delta_{\infty}}f\left(\gamma\right)f\left(\alpha\right)\approx_{\delta_{\Delta}}f\left(\gamma\right)\,\,\text{.}
\]
\end{proof}

\section{\label{sec:special-linear} Flexible stability of special linear groups}

Here we prove a more general version of Theorem \ref{thm:intro-SLr}
about the uniform flexible stability of $\SL_{r}A$, $r\geq3$, where
$A$ is either $\ZZ$ or one of many other commutative rings, as discussed
below. For many of those rings, our result applies to $\SL_{2}A$
as well. We follow the method of \cite[Section 5]{BOT} and use Theorems
\ref{thm:intro-amenable-positive} and \ref{thm:almost-trivial-on-finite-index}
together with a well-known theorem about bounded generation \cite{WitteMorris}.

Fron now on, let $A=S^{-1}B$, where $B$ is an order in the ring
of integers $\calO_{K}$ of an algebraic number field $K/\QQ$ and
$S$ is a multiplicative subset of $B$. For example, we can take
$A=\calO_{K}$. Fix $r\geq2$ and assume that at least one of the
following holds:
\begin{itemize}
\item $r\geq3$;
\item $A$ has infinitely many units.
\end{itemize}
The following result is a special case of \cite[Theorem 6.1]{WitteMorris}.
\begin{thm}
\label{thm:bounded-generation}Let $S$ be a conjugation-invariant
subset of $\SL_{r}\left(A\right)$ that contains at least one non-scalar
matrix. Then $\langle S\rangle$ is a finite-index normal subgroup
of $\SL_{r}\left(A\right)$ and there is an integer $C_{r}$, depending
only on $r$, such that every element of the group $\langle S\rangle$
is a product of at most $C_{r}$ elements of $S$ and their inverses.
\end{thm}

Let $C_{r}$ be the constant provided by the Theorem \ref{thm:bounded-generation}
(we fixed $r$, but we keep it in the notation for emphasis). The
notation $O\left(F\left(x\right)\right)$ is used in Theorem \ref{thm:special-linear}
to refer to an unspecified real-valued function $g\left(x\right)$
such that $\left|g\left(x\right)\right|\leq M\cdot F\left(x\right)$
for all $x\geq0$, where $M$ is an unspecified absolute constant.
We also use the notation $\approx$ as in (\ref{approx_notation}).
\begin{thm}
\label{thm:special-linear}The group $\SL_{r}\left(A\right)$ is uniformly
flexibly stable with linear rate. More explicitly, let $f\colon\SL_{r}\left(A\right)\rightarrow\Sym\left(n\right)$
be a function and write $\delta=\defect_{\infty}\left(f\right)$.
Then there is $n\leq N\leq\left(1+O\left(\delta\right)\right)n$ and
a homomorphism $h\colon\SL_{r}\left(A\right)\rightarrow\Sym\left(N\right)$
such that $d_{\infty}\left(h,f\right)\leq O\left(C_{r}\delta\right)$.
\end{thm}

\begin{proof}
Let $U^{+}$ and $U^{-}$ be the subgroups of $\SL_{r}\left(A\right)$
of upper and lower triangular unipotent matrices, respectively. Both
$U^{+}$ and $U^{-}$ are nilpotent, and thus they are amenable. Apply
Theorem \ref{thm:intro-amenable-positive} to the restrictions $f\mid_{U^{+}}$
and $f\mid_{U^{-}}$. The theorem provides $N_{1},N_{2}\in\NN$, $n\leq N_{i}\leq\left(1+O\left(\delta\right)\right)n$,
and homomorphisms $g^{+}\colon U^{+}\rightarrow\Sym\left(N_{1}\right)$
and $g^{-}\colon U^{-}\rightarrow\Sym\left(N_{2}\right)$ such that
$d_{\infty}\left(g^{+},f\mid_{U^{+}}\right)\leq O\left(\delta\right)$
and $d_{\infty}\left(g^{-},f\mid_{U^{-}}\right)\leq O\left(\delta\right)$.
Let $N=\max\left\{ N_{1},N_{2}\right\} $.

For distinct $i,j\in\left[r\right]$, let $E_{ij}$ be the $r\times r$
matrix with $1$ in the $\left(i,j\right)$ entry and $0$ elsewhere,
and let $U_{ij}=I_{r}+E_{ij}$. We consider the elements $\left\{ U_{ij}\right\} _{i\neq j}$
and their powers. Let $i\neq j$. If $i<j$, then 
\[
f\left(U_{ij}^{\pm N!}\right)\approx_{O\left(\delta\right)}g^{+}\left(U_{ij}^{\pm N!}\right)=\id_{\Sym\left(N_{1}\right)}\approx_{O\left(\delta\right)}\id_{\Sym\left(n\right)}\,\,\text{.}
\]
Similarly, $f\left(U_{ij}^{\pm N!}\right)\approx_{O\left(\delta\right)}\id_{\Sym\left(n\right)}$
for $i>j$. Hence, for $\gamma\in\SL_{r}\left(A\right)$ and $i\neq j$,
we have
\[
f\left(\gamma U_{ij}^{\pm N!}\gamma^{-1}\right)\approx_{O\left(\delta\right)}f\left(\gamma\right)f\left(U_{ij}^{\pm N!}\right)f\left(\gamma^{-1}\right)\approx_{O\left(\delta\right)}f\left(\gamma\right)f\left(\gamma^{-1}\right)\approx_{2\delta}\id_{\Sym\left(n\right)}
\]
(see (\ref{eq:symmetrization-almost-multiplicative}) for the last
step). Let $S=\left\{ \gamma U_{ij}^{\pm N!}\gamma^{-1}\mid i\neq j,\gamma\in\SL_{r}\left(A\right)\right\} $
and $\Delta=\langle S\rangle$. By Theorem \ref{thm:bounded-generation},
$\Delta$ is a finite-index normal subgroup of $\SL_{r}\left(A\right)$
and every element of $\Delta$ is a product of at most $C_{r}$ elements
of $S$. Hence,
\[
f\left(\gamma\right)\approx_{O\left(C_{r}\delta\right)}\id_{\Sym\left(n\right)}\qquad\forall\gamma\in\Delta\,\,\text{.}
\]
The claim now follows from Theorem \ref{thm:almost-trivial-on-finite-index}.
\end{proof}
\begin{rem}
Fix $r\geq2$ and let $\calC$ be the set consisting of all groups
$\SL_{r}\left(A'\right)$ such that the pair $\left(A',r\right)$
satisfies the conditions from the beginning of the section. Then Theorem
\ref{thm:special-linear} says is that the class $\calC$ is uniformly
flexibly stable with linear rate.
\end{rem}

\section{\label{sec:non-strict}Counterexamples for strict stability: the integers
and the class of finite groups}

This section is devoted to the proof of Theorem \ref{thm:intro-not-strictly-stable}
and its probabilistic version. Both versions are included in the statement
of Theorem \ref{thm:drop-a-point} below.

For the sake of the proof of the probabilistic version, we collect
preliminary facts regarding integration on a space equipped with a
finitely-additive probability measure (or \emph{measure} for short,
see Section \ref{subsec:amenable}). 
In our case, the space is an
amenable group $\Gamma$ equipped with a left- or right- invariant
measure $m$, but the preliminary facts hold regardless of the invariance
property.

Recall that integration is a positive linear functional
$L^{\infty}\left(\Gamma\right)\rightarrow\CC$. 
That is, if the image
of $f\in L^{\infty}\left(\Gamma\right)$ is contained in $\RR_{\geq0}$,
then $\int f\dm\geq0$. 
By the proof of \cite[IV.4.1]{DunfordSchwartz},
this is enough for the Cauchy--Schwarz inequality to hold: $\left|\int f_{1}\overline{f_{2}}\dm\right|^{2}\le\left(\int\left|f_{1}\right|^{2}\dm\right)\left(\int\left|f_{2}\right|^{2}\dm\right)$
for bounded functions $f_{1},f_{2}\colon\Gamma\rightarrow\CC$. 
By taking $g={\bf 1}_\Gamma$ we deduce that $\left| \int f\dm  \right|^2\leq \int\left|  f \right|^2\dm $.
For a bounded vector-valued function $f\colon\Gamma\rightarrow\CC^{d}$, define
$\int f\dm$ by integrating coordinatewise. 
The aforementioned corollary of the Cauchy--Schwarz inequality
extends to this setting:  $\|\int f\dm\|^{2}\leq\int\|f\|^{2}\dm$, where $\|\cdot\|$
is the $L^{2}$ norm on $\CC^{d}$. Finally, for a linear operator
$A\colon\CC^{d}\rightarrow\CC^{d}$, we have $A\int f\dm=\int Af\dm$. 

Write $\Psi_{n}\colon\Sym\left(n\right)\rightarrow\Sym\left(n-1\right)$
for the map given by 
\[
\Psi_{n}\left(\sigma\right)\left(x\right)=\begin{cases}
\sigma\left(x\right) & \sigma\left(x\right)\neq n\\
\sigma\left(\sigma\left(x\right)\right) & \sigma\left(x\right)=n
\end{cases}\qquad\forall\sigma\in\Sym\left(n\right)\forall x\in\left[n-1\right]\,\,\text{.}
\]
For a group $\Gamma$ and a homomorphism $f\colon\Gamma\rightarrow\Sym\left(n\right)$,
define $\hat{f}=\Psi_{n}\circ f$. First, note that 
\begin{equation}
\defect_{\infty}\left(\hat{f}\right)\leq\frac{2}{n-1}\,\,\text{,}\label{eq:negative-strict-small-local-defect}
\end{equation}
a fortiori, $\defect_{1}\left(\hat{f}\right)\leq\frac{2}{n-1}$ if
$\Gamma$ is equipped with a measure. Indeed, for $\gamma_{1},\gamma_{2}\in\Gamma$
and $x\in\left[n-1\right]$, if $x\notin\left\{ f\left(\gamma_{2}\right)^{-1}\left(n\right),f\left(\gamma_{1}\gamma_{2}\right)^{-1}\left(n\right)\right\} $,
then $\hat{f}\left(\gamma_{1}\right)\hat{f}\left(\gamma_{2}\right)\left(x\right)=\hat{f}\left(\gamma_{1}\gamma_{2}\right)\left(x\right)$.
Thus $d^{H}\left(\hat{f}\left(\gamma_{1}\right)\hat{f}\left(\gamma_{2}\right),\hat{f}\left(\gamma_{1}\gamma_{2}\right)\right)\leq\frac{2}{n-1}$,
and (\ref{eq:negative-strict-small-local-defect}) follows. In Theorem
\ref{thm:drop-a-point} we show that if $f$ defines a transitive
action $\Gamma\curvearrowright\left[n\right]$ then $\hat{f}$ is
far from every homomorphism $\Gamma\rightarrow\Sym\left(n-1\right)$.
The proof of the theorem relies on the following observation. 
\begin{lem}
\label{lem:nearby-invariant-vector}Let $\rho\colon\Gamma\rightarrow\U\left(\calH\right)$
be a unitary representation of a (discrete) group $\Gamma$ on a finite-dimensional
complex Hilbert space $\calH$, $\eps>0$ and $v\in\calH$.\renewcommand{\labelenumi}{\roman{enumi})}
\begin{enumerate}
\item If $\|\rho\left(\gamma\right)v-v\|\leq\eps\|v\|$ for every $\gamma\in\Gamma$,
then there is a $\Gamma$-invariant vector $u\in\calH$ such that
$\|u-v\|\leq\frac{\eps}{\sqrt{2}}\|v\|$.
\item If $\Gamma$ is an amenable group equipped with a right-invariant
measure $m$ such that $\int\|\rho\left(\gamma^{-1}\right)v-v\|^{2}\dm\left(\gamma\right)\leq\epsilon^{2}\|v\|^{2}$,
then the vector $w=\int\rho\left(\gamma^{-1}\right)v\dm\left(\gamma\right)$
is $\Gamma$-invariant and satisfies $\|w-v\|\leq\eps\|v\|$.
\end{enumerate}
\end{lem}

\begin{proof}
(i) The claim follows from the argument presented in \cite[Propositions 1.1.5 and 1.1.9]{BHV}.
Here we recall a part of the argument, which yields a weaker bound.
Let $C$ be the closed convex hull of the orbit $\rho\left(\Gamma\right)v$
of $v$. Then $C$ is contained in the closed ball of radius $\eps\|v\|$
centered at $v$, and there is a unique point $u$ in $C$ of minimal
norm. Clearly $\|u-v\|\leq\eps\|v\|$. Furthermore, for $\gamma\in\Gamma$
we have $\rho\left(\gamma\right)\left(\rho\left(\Gamma\right)v\right)=\rho\left(\Gamma\right)v$
and thus $\rho\left(\gamma\right)C=C$. Then $\rho\left(\gamma\right)u=u$
since $\rho\left(\gamma\right)$ is norm preserving.

(ii) The right invariance of $m$ implies that $w$ is $\Gamma$-invariant.
Indeed, for $\gamma_{0}\in\Gamma$,

\[
\rho\left(\gamma_{0}\right)w=\rho\left(\gamma_{0}\right)\int\rho\left(\gamma^{-1}\right)v\dm\left(\gamma\right)=\int\rho\left(\left(\gamma\gamma_{0}^{-1}\right)^{-1}\right)v\dm\left(\gamma\right)=\int\rho\left(\gamma^{-1}\right)v\dm\left(\gamma\right)=w\,\,\text{.}
\]
Furthermore,
\[
\|w-v\|^{2}=\|\int\left(\rho\left(\gamma^{-1}\right)v-v\right)\dm\left(\gamma\right)\|^{2}\leq\int\|\rho\left(\gamma^{-1}\right)v-v\|^{2}\dm\left(\gamma\right)\leq\eps^{2}\|v\|^{2}\,\,\text{.}
\]
\end{proof}
For a finite set $X$, write $e_{x}\colon X\rightarrow\CC$ for the
function given by $e_{x}\left(x'\right)=\begin{cases}
1 & x'=x\\
0 & x'\neq x
\end{cases}$ for $x'\in X$. Write $L^{2}\left(X\right)$ for the finite-dimensional
complex Hilbert space of functions $X\rightarrow\CC$, endowed with
the unique Hermitian product such that $\left\{ e_{x}\right\} _{x\in X}$
is an orthonormal basis. Note that for $x,y\in X$, $\|e_{x}-e_{y}\|^{2}=2\cdot{\bf 1}_{x\neq y}=\begin{cases}
0 & x=y\\
2 & x\neq y
\end{cases}$. For a group action $f\colon\Gamma\to \Sym\left(X\right)$, write $\rho_f:\Gamma\to\U\left(L^2\left(X\right)\right)$ for the unitary representation given by $\rho_f\left(\gamma\right)e_x=e_{f\left(\gamma\right) x}$ for $x\in X$.

For finite sets $X$ and $Y$, $x\in X$ and $y\in Y$, write $E_{x,y}\colon L^{2}\left(X\right)\rightarrow L^{2}\left(Y\right)$
for the linear map such that $E_{x,y}\left(e_{x'}\right)=\begin{cases}
e_{y} & x=x'\\
0 & x\neq x'
\end{cases}$ for all $x'\in X$. We make the space $\Lin_{X,Y}\coloneqq\hom_{\CC}\left(L^{2}\left(X\right),L^{2}\left(Y\right)\right)$
of linear maps into a complex Hilbert space by endowing it with the
unique Hermitian product such that $\left\{ E_{x,y}\mid x\in X,y\in Y\right\} $
is an orthonormal basis. 
A pair of group actions $h\colon\Gamma\rightarrow\Sym\left(X\right)$ and $f\colon\Gamma\rightarrow\Sym\left(Y\right)$ gives rise to  a unitary representation $\rho_{h,f}\colon\Gamma\rightarrow\U\left(\Lin_{X,Y}\right)$
defined by
\[
\rho_{h,f}\left(\gamma\right)T=\rho_{f}\left(\gamma\right)\circ T\circ\rho_{h}\left(\gamma^{-1}\right)\qquad\forall\gamma\in\Gamma\forall T\in\Lin_{X,Y}\,\,\text{.}
\]

For $n\in\NN$, write $T_{n-1}\colon L^{2}\left(\left[n-1\right]\right)\rightarrow L^{2}\left(\left[n\right]\right)$
for the linear extension of the inclusion map $\left[n-1\right]\hookrightarrow\left[n\right]$.
Then $\|T_{n-1}\|^{2}=n-1$. For $x\in\left[n-1\right]$, the notation
$e_{x}$ can be used both for a function $\left[n-1\right]\rightarrow\CC$
and for a function $\left[n\right]\rightarrow\CC$. The domain should
be understood from the context. In particular, we write $T_{n-1}\left(e_{x}\right)=e_{x}$.

The proof of Theorem \ref{thm:drop-a-point} makes use of the following
lemma.
\begin{lem}
\label{lem:PropertyT-transitive-isolated}\cite[Proposition 2.4(ii)]{BeckerLubotzky}
Let $\Gamma$ be a group. Take $n\geq 2$ and group homomorphisms
$h\colon\Gamma\rightarrow\Sym\left(n-1\right)$ and
$f\colon\Gamma\rightarrow\Sym\left(n\right)$
such that $f$ defines a transitive action of $\Gamma$ on $\left[n\right]$.
Endow $L^2\left(\left[n-1\right]\right)$ and $L^2\left(\left[n\right]\right)$, respectively, with the representations $\rho_h$ and $\rho_f$.
Then $\|T_{n-1}-T'\|\geq\frac{1}{\sqrt{2}}\|T_{n-1}\|$ for every
morphism of representations $T'\colon L^{2}\left(\left[n-1\right]\right)\rightarrow L^{2}\left(\left[n\right]\right)$.
\end{lem}

\begin{thm}
\label{thm:drop-a-point}Let $\Gamma$ be a group and $f\colon\Gamma\rightarrow\Sym\left(n\right)$,
$n\geq2$, a homomorphism that defines a transitive action $\Gamma\curvearrowright\left[n\right]$.
Let $h\colon\Gamma\rightarrow\Sym\left(n-1\right)$ be an arbitrary
homomorphism. Then $\defect_{\infty}\left(\hat{f}\right)\leq\frac{2}{n-1}$
and $d_{\infty}\left(h,\hat{f}\right)\geq\frac{1}{2}-\frac{1}{n-1}$.

Furthermore, if $\Gamma$ is a discrete amenable group equipped with
a right-invariant measure $m$, then $\defect_{1}\left(\hat{f}\right)\leq\frac{2}{n-1}$
and $d_{1}\left(h,\hat{f}\right)\geq\frac{1}{4}-\frac{1}{n-1}$.
\end{thm}

\begin{proof}
By (\ref{eq:negative-strict-small-local-defect}), $\defect_{\infty}\left(\hat{f}\right)\leq\frac{2}{n-1}$.
Now, we have an action $h\colon\Gamma\rightarrow\Sym\left(\left[n-1\right]\right)$
and a transitive action $f\colon\Gamma\rightarrow\Sym\left(\left[n\right]\right)$.
Consider the representation $\rho\coloneqq\rho_{h,f}\colon\Gamma\rightarrow\U\left(\Lin_{\left[n-1\right],\left[n\right]}\right)$.
For every $\gamma\in\Gamma$,
\begin{align}
\|\rho\left(\gamma^{-1}\right)T_{n-1}-T_{n-1}\|^{2} & =\sum_{x\in\left[n-1\right]}\|\left(\rho_{f}\left(\gamma^{-1}\right)\circ T_{n-1}\circ\rho_{h}\left(\gamma\right)-T_{n-1}\right)e_{x}\|^{2}\nonumber \\
 & =\sum_{x\in\left[n-1\right]}\|\left(T_{n-1}\circ\rho_{h}\left(\gamma\right)-\rho_{f}\left(\gamma\right)\circ T_{n-1}\right)e_{x}\|^{2} & \text{\ensuremath{\rho_{f}\left(\gamma\right)} preserves norm}\nonumber \\
 & =\sum_{x\in\left[n-1\right]}\|e_{h\left(\gamma\right)\left(x\right)}-e_{f\left(\gamma\right)\left(x\right)}\|^{2}\nonumber \\
 & =\sum_{x\in\left[n-1\right]}2\cdot{\bf 1}_{h\left(\gamma\right)\left(x\right)\neq f\left(\gamma\right)\left(x\right)}\nonumber \\
 & \leq2\sum_{x\in\left[n-1\right]}\left({\bf 1}_{h\left(\gamma\right)\left(x\right)\neq\hat{f}\left(\gamma\right)\left(x\right)}+{\bf 1}_{\hat{f}\left(\gamma\right)\left(x\right)\neq f\left(\gamma\right)\left(x\right)}\right) & \text{triangle inequality}\nonumber \\
 & \leq2\left(\sum_{x\in\left[n-1\right]}{\bf 1}_{h\left(\gamma\right)\left(x\right)\neq\hat{f}\left(\gamma\right)\left(x\right)}+1\right)\nonumber \\
 & =2\left(d^{H}\left(h\left(\gamma\right),\hat{f}\left(\gamma\right)\right)+\frac{1}{n-1}\right)\|T_{n-1}\|^{2}\,\,\text{.}\label{eq:drop-a-point-computation}
\end{align}
Hence, for $\eps=\left(2\left(d_{\infty}\left(h,\hat{f}\right)+\frac{1}{n-1}\right)\right)^{1/2}$
we have
\[
\|\rho\left(\gamma\right)T_{n-1}-T_{n-1}\|\leq\eps\|T_{n-1}\|\qquad\forall\gamma\in\Gamma\,\,\text{.}
\]
By Lemma \ref{lem:nearby-invariant-vector}(i), there is $T'\in\Lin_{\left[n-1\right],\left[n\right]}$
such that $\|T'-T_{n-1}\|\leq\frac{\eps}{\sqrt{2}}\|T_{n-1}\|$ and
$\rho\left(\gamma\right)T'=T'$ for every $\gamma\in\Gamma$. The
latter condition means that $T'\colon L^{2}\left(\left[n-1\right]\right)\rightarrow L^{2}\left(\left[n\right]\right)$
is a morphism of representations. Therefore, $\eps\geq1$ by Lemma
\ref{lem:PropertyT-transitive-isolated}. Hence $d_{\infty}\left(h,\hat{f}\right)\geq\frac{1}{2}-\frac{1}{n-1}$.

Now, assume that $m$ is a right-invariant measure on $\Gamma$. Then
(\ref{eq:negative-strict-small-local-defect}) implies that $\defect_{1}\left(\hat{f}\right)\leq\frac{2}{n-1}$.
By integrating (\ref{eq:drop-a-point-computation}), we see that for
$\alpha=\sqrt{2}\left(d_{1}\left(h,\hat{f}\right)+\frac{1}{n-1}\right)^{1/2}$,
\[
\int\|\rho\left(\gamma^{-1}\right)T_{n-1}-T_{n-1}\|^{2}\dm\left(\gamma\right)\le2\|T_{n-1}\|^{2}\int\left(d^{H}\left(h\left(\gamma\right),\hat{f}\left(\gamma\right)\right)+\frac{1}{n-1}\right)\dm\left(\gamma\right)=\alpha^{2}\|T_{n-1}\|^{2}\,\,\text{.}
\]
By Lemma \ref{lem:nearby-invariant-vector}(ii), there is $T''\in W$
such that $\|T''-T_{n-1}\|\leq\alpha\|T_{n-1}\|$ and $T''$ is a
morphism of representations. Using Lemma \ref{lem:PropertyT-transitive-isolated}
as before, we see that $\alpha\geq\frac{1}{\sqrt{2}}$ and thus $d_{1}\left(h,\hat{f}\right)\geq\frac{1}{4}-\frac{1}{n-1}$.
\end{proof}
\begin{rem}
Let $\Gamma$ be a group with finite quotients of unbounded cardinality.
Theorem \ref{thm:drop-a-point} implies that $\Gamma$ is not uniformly
strictly stable. By \cite[Theorem 1.4]{BeckerLubotzky}, if we assume
further that $\Gamma$ has property $\ptau$, then $\Gamma$ is not pointwise
strictly stable. Both the uniform and the pointwise versions are proved
by considering $\hat{f}\colon\Gamma\rightarrow\Sym\left(n-1\right)$,
where $f\colon\Gamma\rightarrow\Sym\left(n\right)$ is a transitive
action.
\end{rem}

\section{\label{sec:non-flexible}Counterexamples for flexible stability:
free groups}

This section is devoted to the proof of Theorem \ref{thm: Free_gps_are_not_FUSP-1}.
The proof involves a construction that combines exponent reduction
on words in a free group, inspired by \cite{Rolli}, with a pinched
grid construction, inspired by \cite[Section 5]{BeckerMosheiff}.

Let $F_{2}$ be the free group on $\left\{ x_{1},x_{2}\right\} $.
Let $k$ be a positive integer. For an integer $t\geq0$, let $\overline{t}$
be the unique element of $C_{k}=\left\{ 0,\dotsc,k-1\right\} $ that
is congruent to $t$ modulo $k$. For $t<0$ we let $\overline{t}=-\overline{\left(-t\right)}$.
For example, if $k=5$ then $\overline{7}=2$ and $\overline{-12}=-2$.
Define $\alpha'_{1},\alpha_{2}\in\Sym\left(C_{k}\times C_{k}\right)$
as follows:
\[
\alpha'_{1}\left(\left(i,j\right)\right)=\left(\overline{i+1},j\right)\,\,\text{and}\,\,\alpha_{2}\left(\left(i,j\right)\right)=\left(i,\overline{j+1}\right)\quad\forall\left(i,j\right)\in C_{k}\times C_{k}
\]
and let $\alpha_{1}=\alpha'_{1}\circ\tau$, where $\tau\in\Sym\left(C_{k}\times C_{k}\right)$
is the transposition that swaps $\left(0,0\right)$ and $\left(0,1\right)$.
The actions of $\alpha_{1}$ and $\alpha_{2}$ on $C_{k}\times C_{k}$
are depicted in Figure \ref{fig:free-groups1}.

\begin{figure}
\begin{center}
\begin{adjustwidth*}{0em}{-4em}   
 \begin{tikzpicture}[scale=1.5]
 \node[draw, color=black] (x00) at (0,0) {\tiny $0,0$}; 
\node[draw, color=black] (x01) at (0,0.8) {\tiny $0,1$};  
\node[draw, color=black] (x10) at (1.3,0) {\tiny $1,0$}; 
\node[draw, color=black] (x11) at (1.3,0.8) {\tiny $1,1$}; 
\node[draw, color=black] (x60) at (8.5,0) {\tiny $k-1,0$};  
\node[draw, color=black] (x66) at (8.5,4.8) {\tiny $k-1,k-1$}; 
\node[draw, color=black] (x06) at (0,4.8) {\tiny $0,k-1$}; 
\node[draw, color=black] (x16) at (1.3,4.8) {\tiny $1,k-1$}; 
\node[draw, color=black] (x61) at (8.5,0.8) {\tiny $k-1,1$};
\node[draw, color=black] (x62) at (8.5,1.6) {\tiny $k-1,2$}; 
\node[draw, color= black] (x02) at (0,1.6) {\tiny $0,2$}; 
\node[draw, color=black] (x12) at (1.3,1.6) {\tiny $1,2$};
\node[draw, color=black] (x22) at (2.6,1.6) {\tiny $2,2$}; 
\node[draw, color=black] (x52) at (6.8,1.6) {\tiny $k-2,2$}; 
\node[draw, color=black] (x51) at (6.8,0.8) {\tiny $k-2,1$};
\node[draw, color=black] (x50) at (6.8,0) {\tiny $k-2,0$}; 
\node[draw, color=black] (x56) at (6.8,4.8) {\tiny $k-2,k-1$}; 
\node[draw, color=black] (x55) at (6.8,4) {\tiny $k-2,k-2$}; 
\node[draw, color=black] (x65) at (8.5,4) {\tiny $k-1,k-2$}; 
\node[draw, color=black] (x05) at (0,4) {\tiny $0,k-2$}; 
\node[draw, color=black] (x15) at (1.3,4) {\tiny $1,k-2$}; 
\node[draw, color=black] (x25) at (2.6,4) {\tiny $2,k-2$}; 
\node[draw, color=black] (x26) at (2.6,4.8) {\tiny $2,k-1$}; 
\node[draw, color=black] (x20) at (2.6,0) {\tiny $2,0$}; 
\node[draw, color=black] (x21) at (2.6,0.8) {\tiny $2,1$}; 
\node[draw, color=black] (x40) at (5.2,0) {\tiny $k-3,0$}; 
\node[draw, color=black] (x41) at (5.2,0.8) {\tiny $k-3,1$}; 
\node[draw, color=black] (x42) at (5.2,1.6) {\tiny $k-3,2$};
\node[draw, color=black] (x44) at (5.2,3.2) {\tiny $k-3,k-3$}; 
\node[draw, color=black] (x45) at (5.2,4) {\tiny $k-3,k-2$}; 
\node[draw, color=black] (x46) at (5.2,4.8) {\tiny $k-3,k-1$}; 
\node[draw, color=black] (x04) at (0,3.2) {\tiny $0,k-3$}; 
\node[draw, color=black] (x14) at (1.3,3.2) {\tiny $1,k-3$}; 
\node[draw, color=black] (x24) at (2.6,3.2) {\tiny $2,k-3$}; 
\node[draw, color=black] (x54) at (6.8,3.2) {\tiny $k-2,k-3$}; 
\node[draw, color=black] (x64) at (8.5,3.2) {\tiny $k-1,k-3$}; 
\node[draw=none, color=black] (x30) at (3.9,0) { $\dots$}; 
\node[draw=none, color=black] (x31) at (3.9,0.8) { $\dots$}; 
\node[draw=none, color=black] (x32) at (3.9,1.6) { $\dots$}; 
\node[draw=none, color=black] (x35) at (3.9,4) { $\dots$}; 
\node[draw=none, color=black] (x36) at (3.9,4.8) { $\dots$};
\node[draw=none, color=black] (x34) at (3.9,3.2) { $\dots$}; 
\node[draw=none, color=black] (x03) at (0,2.4) { $\vdots$}; 
\node[draw=none, color=black] (x13) at (1.3,2.4) { $\vdots$}; 
\node[draw=none, color=black] (x23) at (2.6,2.4) { $\vdots$}; 
\node[draw=none, color=black] (x43) at (5.2,2.4) { $\vdots$}; 
\node[draw=none, color=black] (x53) at (6.8,2.4) { $\vdots$}; 
\node[draw=none, color=black] (x63) at (8.5,2.4) { $\vdots$};

\draw[purple,->, solid] (x11)--(x21) node[midway,above]{\scriptsize$\alpha_1$}; \draw[purple,->, solid] (x01)--(x10) node[midway,above]{\scriptsize$\alpha_1$}; \draw[purple,->, solid] (x61) to  [out=15,in=165,looseness=0.5] (x01) ;  \draw[purple,->, solid] (x21)--(x31) node[midway,above]{\scriptsize$\alpha_1$}; \draw[purple,->, solid] (x51)--(x61) node[midway,above]{\scriptsize$\alpha_1$}; \draw[purple,->, solid] (x51)--(x61) node[midway,above]{\scriptsize$\alpha_1$}; \draw[purple,->, solid] (x41)--(x51) node[midway,above]{\scriptsize$\alpha_1$}; \draw[purple,->, solid] (x31)--(x41) node[midway,above]{\scriptsize$\alpha_1$};
\draw[purple,->, solid] (x10)--(x20) node[midway,above]{\scriptsize$\alpha_1$}; \draw[purple,->, solid] (x00)--(x11) node[midway,above]{\scriptsize$\alpha_1$}; \draw[purple,->, solid] (x60) to  [out=15,in=165,looseness=0.5] (x00) ; \draw[purple,->, solid] (x20)--(x30) node[midway,above]{\scriptsize$\alpha_1$}; \draw[purple,->, solid] (x50)--(x60) node[midway,above]{\scriptsize$\alpha_1$}; \draw[purple,->, solid] (x50)--(x60) node[midway,above]{\scriptsize$\alpha_1$}; \draw[purple,->, solid] (x40)--(x50) node[midway,above]{\scriptsize$\alpha_1$}; \draw[purple,->, solid] (x30)--(x40) node[midway,above]{\scriptsize$\alpha_1$};
\draw[purple,->, solid] (x12)--(x22) node[midway,above]{\scriptsize$\alpha_1$}; \draw[purple,->, solid] (x02)--(x12) node[midway,above]{\scriptsize$\alpha_1$}; \draw[purple,->, solid] (x62) to  [out=15,in=165,looseness=0.5] (x02)  ; \draw[purple,->, solid] (x22)--(x32) node[midway,above]{\scriptsize$\alpha_1$}; \draw[purple,->, solid] (x52)--(x62) node[midway,above]{\scriptsize$\alpha_1$}; \draw[purple,->, solid] (x52)--(x62) node[midway,above]{\scriptsize$\alpha_1$}; \draw[purple,->, solid] (x42)--(x52) node[midway,above]{\scriptsize$\alpha_1$}; \draw[purple,->, solid] (x32)--(x42) node[midway,above]{\scriptsize$\alpha_1$};
\draw[purple,->, solid] (x16)--(x26) node[midway,above]{\scriptsize$\alpha_1$}; \draw[purple,->, solid] (x06)--(x16) node[midway,above]{\scriptsize$\alpha_1$}; \draw[purple,->, solid] (x66) to  [out=15,in=165,looseness=0.5] (x06)  ; \draw[purple,->, solid] (x26)--(x36) node[midway,above]{\scriptsize$\alpha_1$}; \draw[purple,->, solid] (x46)--(x56) node[midway,above]{\scriptsize$\alpha_1$}; \draw[purple,->, solid] (x56)--(x66) node[midway,above]{\scriptsize$\alpha_1$}; \draw[purple,->, solid] (x46)--(x56) node[midway,above]{\scriptsize$\alpha_1$}; \draw[purple,->, solid] (x36)--(x46) node[midway,above]{\scriptsize$\alpha_1$};
\draw[purple,->, solid] (x15)--(x25) node[midway,above]{\scriptsize$\alpha_1$}; \draw[purple,->, solid] (x05)--(x15) node[midway,above]{\scriptsize$\alpha_1$}; \draw[purple,->, solid] (x65) to  [out=15,in=165,looseness=0.5] (x05)  ; \draw[purple,->, solid] (x25)--(x35) node[midway,above]{\scriptsize$\alpha_1$}; \draw[purple,->, solid] (x45)--(x55) node[midway,above]{\scriptsize$\alpha_1$}; \draw[purple,->, solid] (x55)--(x65) node[midway,above]{\scriptsize$\alpha_1$}; \draw[purple,->, solid] (x45)--(x55) node[midway,above]{\scriptsize$\alpha_1$}; \draw[purple,->, solid] (x35)--(x45) node[midway,above]{\scriptsize$\alpha_1$};
\draw[purple,->, solid] (x14)--(x24) node[midway,above]{\scriptsize$\alpha_1$}; \draw[purple,->, solid] (x04)--(x14) node[midway,above]{\scriptsize$\alpha_1$}; \draw[purple,->, solid] (x64) to  [out=15,in=165,looseness=0.5] (x04)  ; \draw[purple,->, solid] (x24)--(x34) node[midway,above]{\scriptsize$\alpha_1$}; \draw[purple,->, solid] (x44)--(x54) node[midway,above]{\scriptsize$\alpha_1$}; \draw[purple,->, solid] (x54)--(x64) node[midway,above]{\scriptsize$\alpha_1$}; \draw[purple,->, solid] (x44)--(x54) node[midway,above]{\scriptsize$\alpha_1$}; \draw[purple,->, solid] (x34)--(x44) node[midway,above]{\scriptsize$\alpha_1$};
\draw[cyan,->, dashed] (x05)--(x06) node[midway,right]{\scriptsize$\alpha_2$}; \draw[cyan,->, dashed] (x06) to  [out=255,in=105,looseness=1] (x00)  ; \draw[cyan,->, dashed] (x00)--(x01) node[midway,right]{\scriptsize$\alpha_2$}; \draw[cyan,->, dashed] (x01)--(x02) node[midway,right]{\scriptsize$\alpha_2$}; \draw[cyan,->, dashed] (x02)--(x03) node[midway,right]{\scriptsize$\alpha_2$}; \draw[cyan,->, dashed] (x03)--(x04) node[midway,right]{\scriptsize$\alpha_2$}; \draw[cyan,->, dashed] (x04)--(x05) node[midway,right]{\scriptsize$\alpha_2$};
\draw[cyan,->, dashed] (x15)--(x16) node[midway,right]{\scriptsize$\alpha_2$}; \draw[cyan,->, dashed] (x16) to  [out=255,in=105,looseness=1] (x10)  ; \draw[cyan,->, dashed] (x10)--(x11) node[midway,right]{\scriptsize$\alpha_2$}; \draw[cyan,->, dashed] (x11)--(x12) node[midway,right]{\scriptsize$\alpha_2$}; \draw[cyan,->, dashed] (x12)--(x13) node[midway,right]{\scriptsize$\alpha_2$}; \draw[cyan,->, dashed] (x13)--(x14) node[midway,right]{\scriptsize$\alpha_2$}; \draw[cyan,->, dashed] (x14)--(x15) node[midway,right]{\scriptsize$\alpha_2$};
\draw[cyan,->, dashed] (x25)--(x26) node[midway,right]{\scriptsize$\alpha_2$}; \draw[cyan,->, dashed] (x26) to  [out=255,in=105,looseness=1] (x20)  ; \draw[cyan,->, dashed] (x20)--(x21) node[midway,right]{\scriptsize$\alpha_2$}; \draw[cyan,->, dashed] (x21)--(x22) node[midway,right]{\scriptsize$\alpha_2$}; \draw[cyan,->, dashed] (x22)--(x23) node[midway,right]{\scriptsize$\alpha_2$}; \draw[cyan,->, dashed] (x23)--(x24) node[midway,right]{\scriptsize$\alpha_2$}; \draw[cyan,->, dashed] (x24)--(x25) node[midway,right]{\scriptsize$\alpha_2$};
\draw[cyan,->, dashed] (x65)--(x66) node[midway,right]{\scriptsize$\alpha_2$}; \draw[cyan,->, dashed] (x66) to  [out=255,in=105,looseness=1] (x60)  ; \draw[cyan,->, dashed] (x60)--(x61) node[midway,right]{\scriptsize$\alpha_2$}; \draw[cyan,->, dashed] (x61)--(x62) node[midway,right]{\scriptsize$\alpha_2$}; \draw[cyan,->, dashed] (x62)--(x63) node[midway,right]{\scriptsize$\alpha_2$}; \draw[cyan,->, dashed] (x63)--(x64) node[midway,right]{\scriptsize$\alpha_2$}; \draw[cyan,->, dashed] (x64)--(x65) node[midway,right]{\scriptsize$\alpha_2$};
\draw[cyan,->, dashed] (x55)--(x56) node[midway,right]{\scriptsize$\alpha_2$}; \draw[cyan,->, dashed] (x56) to  [out=255,in=105,looseness=1] (x50)  ; \draw[cyan,->, dashed] (x50)--(x51) node[midway,right]{\scriptsize$\alpha_2$}; \draw[cyan,->, dashed] (x51)--(x52) node[midway,right]{\scriptsize$\alpha_2$}; \draw[cyan,->, dashed] (x52)--(x53) node[midway,right]{\scriptsize$\alpha_2$}; \draw[cyan,->, dashed] (x53)--(x54) node[midway,right]{\scriptsize$\alpha_2$}; \draw[cyan,->, dashed] (x54)--(x55) node[midway,right]{\scriptsize$\alpha_2$};
\draw[cyan,->, dashed] (x45)--(x46) node[midway,right]{\scriptsize$\alpha_2$}; \draw[cyan,->, dashed] (x46) to  [out=255,in=105,looseness=1] (x40)  ; \draw[cyan,->, dashed] (x40)--(x41) node[midway,right]{\scriptsize$\alpha_2$}; \draw[cyan,->, dashed] (x41)--(x42) node[midway,right]{\scriptsize$\alpha_2$}; \draw[cyan,->, dashed] (x42)--(x43) node[midway,right]{\scriptsize$\alpha_2$}; \draw[cyan,->, dashed] (x43)--(x44) node[midway,right]{\scriptsize$\alpha_2$}; \draw[cyan,->, dashed] (x44)--(x45) node[midway,right]{\scriptsize$\alpha_2$};
\end{tikzpicture}
\end{adjustwidth*}
\caption{The actions of {\color{purple} $\alpha_1$ (solid)} and {\color{cyan} $\alpha_2$ (dashed)} on $C_k\times C_k$.}
\label{fig:free-groups1}
\vspace*{\floatsep}
\vspace*{\floatsep}
\vspace*{\floatsep}

\begin{adjustwidth*}{0em}{-4em}   
	 \begin{tikzpicture}[scale=1.5]  
	 \node[draw, color=black] (x00) at (0,0) {\tiny $0,0$}; 
	 \node[draw, color=black] (x01) at (0,0.8) {\tiny $0,1$};  
	 \node[draw, color=black] (x10) at (1.3,0) {\tiny $1,0$}; 
	 \node[draw, color=black] (x11) at (1.3,0.8) {\tiny $1,1$}; 
	 \node[draw, color=black] (x60) at (8.5,0) {\tiny $k-1,0$};  
	 \node[draw, color=black] (x66) at (8.5,4.8) {\tiny $k-1,k-1$}; 
	 \node[draw, color=black] (x06) at (0,4.8) {\tiny $0,k-1$}; 
	 \node[draw, color=black] (x16) at (1.3,4.8) {\tiny $1,k-1$}; 
	 \node[draw, color=black] (x61) at (8.5,0.8) {\tiny $k-1,1$};
	 \node[draw, color=black] (x62) at (8.5,1.6) {\tiny $k-1,2$}; 
	 \node[draw, color= black] (x02) at (0,1.6) {\tiny $0,2$}; 
	 \node[draw, color=black] (x12) at (1.3,1.6) {\tiny $1,2$};
	 \node[draw, color=black] (x22) at (2.6,1.6) {\tiny $2,2$}; 
	 \node[draw, color=black] (x52) at (6.8,1.6) {\tiny $k-2,2$}; 
	 \node[draw, color=black] (x51) at (6.8,0.8) {\tiny $k-2,1$};
	 \node[draw, color=black] (x50) at (6.8,0) {\tiny $k-2,0$}; 
	 \node[draw, color=black] (x56) at (6.8,4.8) {\tiny $k-2,k-1$}; 
	 \node[draw, color=black] (x55) at (6.8,4) {\tiny $k-2,k-2$}; 
	 \node[draw, color=black] (x65) at (8.5,4) {\tiny $k-1,k-2$}; 
	 \node[draw, color=black] (x05) at (0,4) {\tiny $0,k-2$}; 
	 \node[draw, color=black] (x15) at (1.3,4) {\tiny $1,k-2$}; 
	 \node[draw, color=black] (x25) at (2.6,4) {\tiny $2,k-2$}; 
	 \node[draw, color=black] (x26) at (2.6,4.8) {\tiny $2,k-1$}; 
	 \node[draw, color=black] (x20) at (2.6,0) {\tiny $2,0$}; 
	 \node[draw, color=black] (x21) at (2.6,0.8) {\tiny $2,1$}; 
	 \node[draw, color=black] (x40) at (5.2,0) {\tiny $k-3,0$}; 
	 \node[draw, color=black] (x41) at (5.2,0.8) {\tiny $k-3,1$}; 
	 \node[draw, color=black] (x42) at (5.2,1.6) {\tiny $k-3,2$};
	 \node[draw, color=black] (x44) at (5.2,3.2) {\tiny $k-3,k-3$}; 
	 \node[draw, color=black] (x45) at (5.2,4) {\tiny $k-3,k-2$}; 
	 \node[draw, color=black] (x46) at (5.2,4.8) {\tiny $k-3,k-1$}; 
	 \node[draw, color=black] (x04) at (0,3.2) {\tiny $0,k-3$}; 
	 \node[draw, color=black] (x14) at (1.3,3.2) {\tiny $1,k-3$}; 
	 \node[draw, color=black] (x24) at (2.6,3.2) {\tiny $2,k-3$}; 
	 \node[draw, color=black] (x54) at (6.8,3.2) {\tiny $k-2,k-3$}; 
	 \node[draw, color=black] (x64) at (8.5,3.2) {\tiny $k-1,k-3$}; 
	 \node[draw=none, color=black] (x30) at (3.9,0) { $\dots$}; 
	 \node[draw=none, color=black] (x31) at (3.9,0.8) { $\dots$}; 
	 \node[draw=none, color=black] (x32) at (3.9,1.6) { $\dots$}; 
	 \node[draw=none, color=black] (x35) at (3.9,4) { $\dots$}; 
	 \node[draw=none, color=black] (x36) at (3.9,4.8) { $\dots$};
	 \node[draw=none, color=black] (x34) at (3.9,3.2) { $\dots$}; 
	 \node[draw=none, color=black] (x03) at (0,2.4) { $\vdots$}; 
	 \node[draw=none, color=black] (x13) at (1.3,2.4) { $\vdots$}; 
	 \node[draw=none, color=black] (x23) at (2.6,2.4) { $\vdots$}; 
	 \node[draw=none, color=black] (x43) at (5.2,2.4) { $\vdots$}; 
	 \node[draw=none, color=black] (x53) at (6.8,2.4) { $\vdots$}; 
	 \node[draw=none, color=black] (x63) at (8.5,2.4) { $\vdots$};
	 
\draw[brown,->, solid] (x11)--(x20) node[midway,above]{\scriptsize$\alpha_1^{-k+1}$}; 
\draw[brown,->, solid] (x01)--(x11) node[midway,above]{\scriptsize$\alpha_1^{-k+1}$}; 
\draw[brown,->, solid] (x61) to [out=-15,in=300,looseness=0.4]  (x00) ;  
\draw[brown,->, solid] (x21)--(x30);
\draw[brown,->, solid] (x51)--(x60) node[midway,above]{\scriptsize$\alpha_1^{-k+1}$}; 
\draw[brown,->, solid] (x41)--(x50) node[midway,above]{\scriptsize$\alpha_1^{-k+1}$}; 
\draw[brown,->, solid] (x31)--(x40) node[midway,above]{\scriptsize$\alpha_1^{-k+1}$};  ;
\draw[brown,->, solid] (x10)--(x21) node[midway,above]{\scriptsize$\alpha_1^{-k+1}$};
 \draw[brown,->, solid] (x00)--(x10) node[midway,above]{\scriptsize$\alpha_1^{-k+1}$};
  \draw[brown,->, solid] (x60) to  [out=15,in=60,looseness=0.4] (x01) ; 
  \draw[brown,->, solid] (x20)--(x31) node[midway,above]{\scriptsize$\alpha_1^{-k+1}$}; 
  \draw[brown,->, solid] (x50)--(x61) node[midway,above]{\scriptsize$\alpha_1^{-k+1}$}; 
  \draw[brown,->, solid] (x40)--(x51) node[midway,above]{\scriptsize$\alpha_1^{-k+1}$}; 
  \draw[brown,->, solid] (x30)--(x41) ;
\draw[brown,->, solid] (x12)--(x22) node[midway,above]{\scriptsize$\alpha_1^{-k+1}$};
 \draw[brown,->, solid] (x02)--(x12) node[midway,above]{\scriptsize$\alpha_1^{-k+1}$}; \draw[brown,->, solid] (x62) to  [out=15,in=165,looseness=0.5] (x02)  ; \draw[brown,->, solid] (x22)--(x32) node[midway,above]{\scriptsize$\alpha_1^{-k+1}$}; \draw[brown,->, solid] (x52)--(x62) node[midway,above]{\scriptsize$\alpha_1^{-k+1}$}; \draw[brown,->, solid] (x52)--(x62) node[midway,above]{\scriptsize$\alpha_1^{-k+1}$}; \draw[brown,->, solid] (x42)--(x52) node[midway,above]{\scriptsize$\alpha_1^{-k+1}$}; \draw[brown,->, solid] (x32)--(x42) node[midway,above]{\scriptsize$\alpha_1^{-k+1}$};
\draw[brown,->, solid] (x16)--(x26) node[midway,above]{\scriptsize$\alpha_1^{-k+1}$}; \draw[brown,->, solid] (x06)--(x16) node[midway,above]{\scriptsize$\alpha_1^{-k+1}$}; \draw[brown,->, solid] (x66) to  [out=15,in=165,looseness=0.5] (x06)  ; \draw[brown,->, solid] (x26)--(x36) node[midway,above]{\scriptsize$\alpha_1^{-k+1}$}; \draw[brown,->, solid] (x46)--(x56) node[midway,above]{\scriptsize$\alpha_1^{-k+1}$}; \draw[brown,->, solid] (x56)--(x66) node[midway,above]{\scriptsize$\alpha_1^{-k+1}$}; \draw[brown,->, solid] (x46)--(x56) node[midway,above]{\scriptsize$\alpha_1^{-k+1}$}; \draw[brown,->, solid] (x36)--(x46) node[midway,above]{\scriptsize$\alpha_1^{-k+1}$};
\draw[brown,->, solid] (x15)--(x25) node[midway,above]{\scriptsize$\alpha_1^{-k+1}$}; \draw[brown,->, solid] (x05)--(x15) node[midway,above]{\scriptsize$\alpha_1^{-k+1}$}; \draw[brown,->, solid] (x65) to  [out=15,in=165,looseness=0.5] (x05)  ; \draw[brown,->, solid] (x25)--(x35) node[midway,above]{\scriptsize$\alpha_1^{-k+1}$}; \draw[brown,->, solid] (x45)--(x55) node[midway,above]{\scriptsize$\alpha_1^{-k+1}$}; \draw[brown,->, solid] (x55)--(x65) node[midway,above]{\scriptsize$\alpha_1^{-k+1}$}; \draw[brown,->, solid] (x45)--(x55) node[midway,above]{\scriptsize$\alpha_1^{-k+1}$}; \draw[brown,->, solid] (x35)--(x45) node[midway,above]{\scriptsize$\alpha_1^{-k+1}$};
\draw[brown,->, solid] (x14)--(x24) node[midway,above]{\scriptsize$\alpha_1^{-k+1}$}; \draw[brown,->, solid] (x04)--(x14) node[midway,above]{\scriptsize$\alpha_1^{-k+1}$}; \draw[brown,->, solid] (x64) to  [out=15,in=165,looseness=0.5] (x04)  ; \draw[brown,->, solid] (x24)--(x34) node[midway,above]{\scriptsize$\alpha_1^{-k+1}$}; \draw[brown,->, solid] (x44)--(x54) node[midway,above]{\scriptsize$\alpha_1^{-k+1}$}; \draw[brown,->, solid] (x54)--(x64) node[midway,above]{\scriptsize$\alpha_1^{-k+1}$}; \draw[brown,->, solid] (x44)--(x54) node[midway,above]{\scriptsize$\alpha_1^{-k+1}$}; \draw[brown,->, solid] (x34)--(x44) node[midway,above]{\scriptsize$\alpha_1^{-k+1}$};
\draw[cyan,->, dashed] (x05)--(x06) node[midway,right]{\scriptsize$\alpha_2$}; \draw[cyan,->, dashed] (x06) to  [out=255,in=105,looseness=1] (x00)  ; \draw[cyan,->, dashed] (x00)--(x01) node[midway,right]{\scriptsize$\alpha_2$}; \draw[cyan,->, dashed] (x01)--(x02) node[midway,right]{\scriptsize$\alpha_2$}; \draw[cyan,->, dashed] (x02)--(x03) node[midway,right]{\scriptsize$\alpha_2$}; \draw[cyan,->, dashed] (x03)--(x04) node[midway,right]{\scriptsize$\alpha_2$}; \draw[cyan,->, dashed] (x04)--(x05) node[midway,right]{\scriptsize$\alpha_2$};
\draw[cyan,->, dashed] (x15)--(x16) node[midway,right]{\scriptsize$\alpha_2$}; \draw[cyan,->, dashed] (x16) to  [out=255,in=105,looseness=1] (x10)  ; \draw[cyan,->, dashed] (x10)--(x11) node[midway,right]{\scriptsize$\alpha_2$}; \draw[cyan,->, dashed] (x11)--(x12) node[midway,right]{\scriptsize$\alpha_2$}; \draw[cyan,->, dashed] (x12)--(x13) node[midway,right]{\scriptsize$\alpha_2$}; \draw[cyan,->, dashed] (x13)--(x14) node[midway,right]{\scriptsize$\alpha_2$}; \draw[cyan,->, dashed] (x14)--(x15) node[midway,right]{\scriptsize$\alpha_2$};
\draw[cyan,->, dashed] (x25)--(x26) node[midway,right]{\scriptsize$\alpha_2$}; \draw[cyan,->, dashed] (x26) to  [out=255,in=105,looseness=1] (x20)  ; \draw[cyan,->, dashed] (x20)--(x21) node[midway,right]{\scriptsize$\alpha_2$}; \draw[cyan,->, dashed] (x21)--(x22) node[midway,right]{\scriptsize$\alpha_2$}; \draw[cyan,->, dashed] (x22)--(x23) node[midway,right]{\scriptsize$\alpha_2$}; \draw[cyan,->, dashed] (x23)--(x24) node[midway,right]{\scriptsize$\alpha_2$}; \draw[cyan,->, dashed] (x24)--(x25) node[midway,right]{\scriptsize$\alpha_2$};
\draw[cyan,->, dashed] (x65)--(x66) node[midway,right]{\scriptsize$\alpha_2$}; \draw[cyan,->, dashed] (x66) to  [out=255,in=105,looseness=1] (x60)  ; \draw[cyan,->, dashed] (x60)--(x61) node[midway,right]{\scriptsize$\alpha_2$}; \draw[cyan,->, dashed] (x61)--(x62) node[midway,right]{\scriptsize$\alpha_2$}; \draw[cyan,->, dashed] (x62)--(x63) node[midway,right]{\scriptsize$\alpha_2$}; \draw[cyan,->, dashed] (x63)--(x64) node[midway,right]{\scriptsize$\alpha_2$}; \draw[cyan,->, dashed] (x64)--(x65) node[midway,right]{\scriptsize$\alpha_2$};
\draw[cyan,->, dashed] (x55)--(x56) node[midway,right]{\scriptsize$\alpha_2$}; \draw[cyan,->, dashed] (x56) to  [out=255,in=105,looseness=1] (x50)  ; \draw[cyan,->, dashed] (x50)--(x51) node[midway,right]{\scriptsize$\alpha_2$}; \draw[cyan,->, dashed] (x51)--(x52) node[midway,right]{\scriptsize$\alpha_2$}; \draw[cyan,->, dashed] (x52)--(x53) node[midway,right]{\scriptsize$\alpha_2$}; \draw[cyan,->, dashed] (x53)--(x54) node[midway,right]{\scriptsize$\alpha_2$}; \draw[cyan,->, dashed] (x54)--(x55) node[midway,right]{\scriptsize$\alpha_2$};
\draw[cyan,->, dashed] (x45)--(x46) node[midway,right]{\scriptsize$\alpha_2$}; \draw[cyan,->, dashed] (x46) to  [out=255,in=105,looseness=1] (x40)  ; \draw[cyan,->, dashed] (x40)--(x41) node[midway,right]{\scriptsize$\alpha_2$}; \draw[cyan,->, dashed] (x41)--(x42) node[midway,right]{\scriptsize$\alpha_2$}; \draw[cyan,->, dashed] (x42)--(x43) node[midway,right]{\scriptsize$\alpha_2$}; \draw[cyan,->, dashed] (x43)--(x44) node[midway,right]{\scriptsize$\alpha_2$}; \draw[cyan,->, dashed] (x44)--(x45) node[midway,right]{\scriptsize$\alpha_2$};
\end{tikzpicture}
\end{adjustwidth*}
\vspace*{-15mm}
\caption{The actions of {\color{brown} $\alpha_1^{-k+1}$ (solid)} and {\color{cyan} $\alpha_2$ (dashed)} on $C_k\times C_k$.}
\label{fig:free-groups2}
\end{center}
\end{figure}

Define $g_{k}\colon F_{2}\to\Sym\left(C_{k}\times C_{k}\right)$ as
follows. Let $w\in F_{2}$ be a reduced word. Write $w=x_{1}^{d_{1}}x_{2}^{e_{1}}\cdots x_{1}^{d_{r}}x_{2}^{e_{r}}\in F_{2}$,
$r\ge0$, $d_{i},e_{i}\in\ZZ$, where $d_{i}\neq0$ for $i>1$ and
$e_{i}\neq0$ for $i<r$. Define $g_{k}\left(w\right)=\alpha_{1}^{\overline{d_{1}}}\alpha_{2}^{\overline{e_{1}}}\cdots\alpha_{1}^{\overline{d_{r}}}\alpha_{2}^{\overline{e_{r}}}$.
For example, if $k=5$ then $g_{k}\left(x_{1}^{13}x_{2}^{-9}x_{1}^{3}x_{2}x_{1}^{-77}\right)=\alpha_{1}^{3}\alpha_{2}^{-4}\alpha_{1}^{3}\alpha_{2}\alpha_{1}^{-2}$.

The following lemma shows that $g_{k}$ has small local defect, but
grossly violates an identity that holds in $\Sym\left(N\right)$ for
every $N\geq k^{2}$. The lemma readily implies Theorem \ref{thm: Free_gps_are_not_FUSP-1}
(see below). We shall write $\ell\left(w\right)$ for the length of
a reduced word $w$.
\begin{lem}
\label{lem:lemma-free-group}Let $k\geq1$. Then $\defect_{\infty}\left(g_{k}\right)\leq\frac{2}{k}$,
but $d^{H}\left(g_{k}\left(\left(x_{1}^{N!-k+1}x_{2}\right)^{k}\left(x_{1}^{-k+1}x_{2}\right)^{-k}\right),\id\right)\geq1-\frac{5}{k}$
for every $N\geq k^{2}$.
\end{lem}

\begin{proof}
Write $g=g_{k}$. Let $w_{1},w_{2}\in F_{2}$ be reduced words and
write $\beta_{w_{1},w_{2}}=g\left(w_{1}\right)^{-1}g\left(w_{1}w_{2}\right)g\left(w_{2}\right)^{-1}$.
We prove that $d^{H}\left(\beta_{w_{1},w_{2}},\id\right)\leq\frac{2}{k}$
by induction on $\ell\left(w_{1}\right)+\ell\left(w_{2}\right)$.
If either $w_{1}=\id$ or $w_{2}=\id$ then $\beta_{w_{1},w_{2}}=\id$
and we are done. Assume that $w_{1}\neq\id$ and $w_{2}\neq\id$.
If the last letter of $w_{1}$ and the first letter of $w_{2}$ are
neither the same nor mutual inverses, then $g\left(w_{1}w_{2}\right)=g\left(w_{1}\right)g\left(w_{2}\right)$,
and thus $\beta_{w_{1},w_{2}}=\id$ and we are done. Otherwise, we
have cancellation-free concatenations $w_{1}=\tilde{w}_{1}\cdot x^{e_{1}}$
and $w_{2}=x^{e_{2}}\cdot\tilde{w}_{2}$, where $\tilde{w}_{1}$ and
$\tilde{w}_{2}$ are reduced words, $x\in\left\{ x_{1},x_{2}\right\} $
and $e_{1},e_{2}\in\ZZ\setminus\left\{ 0\right\} $. Write $\alpha=g\left(x\right)\in\left\{ \alpha_{1},\alpha_{2}\right\} $.
If $e_{2}\ne-e_{1}$ then $g\left(w_{1}w_{2}\right)=g\left(\tilde{w}_{1}\right)\alpha^{\overline{e_{1}+e_{2}}}g\left(\tilde{w}_{2}\right)$,
and thus $\beta_{w_{1},w_{2}}=\alpha^{t}$ for $t=\overline{e_{1}+e_{2}}-\left(\overline{e_{1}}+\overline{e_{2}}\right)$.
Then $\left|t\right|\in\left\{ 0,k\right\} $. Since $\alpha_{2}^{k}$
is $\id$ and $\alpha_{1}^{k}$ fixes all elements outside $C_{k}\times\left\{ 0,1\right\} $,
we see that $d^{H}\left(\beta_{w_{1},w_{2}},\id\right)\le\frac{2}{k}$
as required. Finally, if $e_{2}=-e_{1}$, then
\[
\beta_{w_{1},w_{2}}=\left(g\left(\tilde{w}_{1}\right)\alpha^{\overline{e_{1}}}\right)^{-1}g\left(\tilde{w}_{1}\alpha^{\overline{e_{1}}}\cdot\alpha^{\overline{-e_{1}}}\tilde{w}_{2}\right)\left(\alpha^{\overline{-e_{1}}}g\left(\tilde{w}_{2}\right)\right)^{-1}=\alpha^{-\overline{e_{1}}}\beta_{\tilde{w}_{1},\tilde{w}_{2}}\alpha^{\overline{e_{1}}}\,\,\text{,}
\]
and thus $d^{H}\left(\beta_{w_{1},w_{2}},\id\right)=d^{H}\left(\beta_{\tilde{w}_{1},\tilde{w}_{2}},\id\right)\leq\frac{2}{k}$
by the induction hypothesis. The upshot is that $\defect_{\infty}\left(g\right)\leq\frac{2}{k}$
as claimed.

Let $N\geq k^{2}$. Then $\overline{N!-k+1}=1$ and $\overline{-k+1}=-k+1$,
and thus $g\left(\left(x_{1}^{N!-k+1}x_{2}\right)^{k}\right)=\left(\alpha_{1}\alpha_{2}\right)^{k}$
and $g\left(\left(x_{1}^{-k+1}x_{2}\right)^{k}\right)=\left(\alpha_{1}^{-k+1}\alpha_{2}\right)^{k}$.
For $i\in\NN$, write
\[
D_{i}=\left\{ \left(x,y\right)\in C_{k}\times C_{k}\mid x-y\equiv i\pmod k\right\}\,\,\text{.}
\]
By examining Figure \ref{fig:free-groups1}, we see that $\left(\alpha_{1}\alpha_{2}\right)^{k}$
fixes every element of $C_{k}\times C_{k}$ outside the diagonals
$D_{0}$ and $D_{1}$, and thus $d^{H}\left(\left(\alpha_{1}\alpha_{2}\right)^{k},\id\right)\le\frac{2}{k}$.
By examining Figure \ref{fig:free-groups2}, we see that $\left(\alpha_{1}^{-k+1}\alpha_{2}\right)^{k}$
does not fix any element outside the diagonal $D_{0}$, and hence
$d^{H}\left(\left(\alpha_{1}^{-k+1}\alpha_{2}\right)^{k},\id\right)\geq1-\frac{1}{k}$.
Therefore, by the triangle inequality,
\[
d^{H}\left(g\left(\left(x_{1}^{N!-k+1}x_{2}\right)^{k}\right),g\left(\left(x_{1}^{-k+1}x_{2}\right)^{k}\right)\right)\geq1-\frac{3}{k}\ .
\]
Thus, since $d^H$ is bi-invariant and $\defect_{\infty}\left(g\right)\leq\frac{2}{k}$, we have
\[
d^{H}\left(g\left(\left(x_{1}^{N!-k+1}x_{2}\right)^{k}\left(x_{1}^{-k+1}x_{2}\right)^{-k}\right),\id\right)\geq1-\frac{5}{k}\,\,\text{.}
\]
\end{proof}
For $k\geq1$, fix an arbitrary bijection between $\left[k^{2}\right]$
and $C_{k}\times C_{k}$. Henceforth, we use this bijection to view
$g_{k}$ as a function from $\Gamma$ to $\Sym\left(k^{2}\right)$.
\begin{proof}
[Proof of Theorem \ref{thm: Free_gps_are_not_FUSP-1}]We are given
a group $\Gamma$ and a surjective homomorphism $\pi\colon\Gamma\rightarrow F_{2}$.
Define $f_{k}=g_{k}\circ\pi$. Then $\defect_{\infty}\left(f_{k}\right)=\defect_{\infty}\left(g_{k}\right)\leq\frac{2}{k}$.
Take $\gamma_{1},\gamma_{2}\in\Gamma$ such that $\pi\left(\gamma_{1}\right)=x_{1}$
and $\pi\left(\gamma_{2}\right)=x_{2}$.
Let $N\geq k^{2}$, take a homomorphism $h_{k}\colon\Gamma\rightarrow\Sym\left(N\right)$,
and write $\gamma_{0}=\left(\gamma_{1}^{N!-k+1}\gamma_{2}\right)^{k}\left(\gamma_{1}^{-k+1}\gamma_{2}\right)^{-k}\in\Gamma$.
Then $h_{k}\left(\gamma_{0}\right)=\id_{N}\in\Sym\left(N\right)$,
while Lemma \ref{lem:lemma-free-group} says that $d^{H}\left(f_{k}\left(\gamma_{0}\right),\id_{k^{2}}\right)\geq1-\frac{5}{k}$.
Hence, $d_{\infty}\left(f_{k},h_{k}\right)\geq d^{H}\left(f_{k}\left(\gamma_{0}\right),h_{k}\left(\gamma_{0}\right)\right)\geq1-\frac{5}{k}$.
\end{proof}

\appendix

\section{\label{sec:triangle-ineq}The triangle inequality for $d^{H}$}

It is clear that $d^{H}$ is symmetric and that $d^{H}\left(\sigma_{1},\sigma_{2}\right)=0$
if and only if $\sigma_{1}=\sigma_{2}$. In fact, $d^{H}$ is a metric
by the following lemma.
\begin{lem}
\label{lem:triangle-inequality}For $i\in\left[3\right]$, let $n_{i}\geq1$
and $\sigma_{i}\in\Sym\left(n_{i}\right)$. Then $d^{H}\left(\sigma_{1},\sigma_{2}\right)+d^{H}\left(\sigma_{2},\sigma_{3}\right)\geq d^{H}\left(\sigma_{1},\sigma_{3}\right)$.
\end{lem}

\begin{proof}
For $i,j\in\left[3\right]$ write $n_{ij}=\max\left\{ n_{i},n_{j}\right\} $
and $n=\max\left\{ n_{1},n_{2},n_{3}\right\} $. For a permutation
$\sigma\in\Sym\left(N\right)$ and $x>N$ write $\sigma\left(x\right)=\star$,
where $\star$ is a dummy object that is not a natural number. For
$k\geq1$ and permutations $\tau_{1}$ and $\tau_{2}$ (of possibly
different sizes), write $d_{k}\left(\tau_{1},\tau_{2}\right)=\left|\left\{ x\in\left[k\right]\mid\tau_{1}\left(x\right)\neq\tau_{2}\left(x\right)\right\} \right|$.
Then, for $k\geq n_{ij}$ we have $d^{H}\left(\sigma_{i},\sigma_{j}\right)=\frac{1}{n_{ij}}d_{k}\left(\sigma_{i},\sigma_{j}\right)$.
Clearly, $d_{k}$ satisfies the triangle inequality for each fixed
$k$.

Using these notations, we have
\[
d^{H}\left(\sigma_{1},\sigma_{2}\right)+d^{H}\left(\sigma_{2},\sigma_{3}\right)\geq\frac{1}{n}\left(d_{n}\left(\sigma_{1},\sigma_{2}\right)+d_{n}\left(\sigma_{2},\sigma_{3}\right)\right)\geq\frac{1}{n}d_{n_{13}}\left(\sigma_{1},\sigma_{3}\right)\,\,\text{.}
\]
If $n_{2}\leq n_{13}$ then $n=n_{13}$, and thus we are done by the
above. On the other hand, if $n_{2}>n_{13}$ then
\begin{align*}
d^{H}\left(\sigma_{1},\sigma_{2}\right)+d^{H}\left(\sigma_{2},\sigma_{3}\right) & =\frac{1}{n_{2}}\left(d_{n_{2}}\left(\sigma_{1},\sigma_{2}\right)+d_{n_{2}}\left(\sigma_{2},\sigma_{3}\right)\right)\\
 & =\frac{1}{n_{2}}\left(d_{n_{13}}\left(\sigma_{1},\sigma_{2}\right)+d_{n_{13}}\left(\sigma_{2},\sigma_{3}\right)+2\left(n_{2}-n_{13}\right)\right)\\
 & \geq\frac{n_{13}}{n_{2}n_{13}}\left(d_{n_{13}}\left(\sigma_{1},\sigma_{3}\right)+\left(n_{2}-n_{13}\right)\right)\\
 & \geq\frac{1}{n_{2}n_{13}}\left(n_{13}d_{n_{13}}\left(\sigma_{1},\sigma_{3}\right)+\left(n_{2}-n_{13}\right)d_{n_{13}}\left(\sigma_{1},\sigma_{3}\right)\right)\\
 & =\frac{1}{n_{13}}d_{n_{13}}\left(\sigma_{1},\sigma_{3}\right)\\
 & =d^{H}\left(\sigma_{1},\sigma_{3}\right)\,\,\text{.}
\end{align*}
\end{proof}

\section{\label{sec:symmetrization}Symmetrization}

Let $G$ be a group equipped with a bi-invariant metric $d\colon G\times G\rightarrow\RR_{\geq0}$.
Bi-invariance means that $d\left(agb,ahb\right)=d\left(g,h\right)$
for all $g,h,a,b\in G$. For example, one can take $G$ to be $\Sym\left(n\right)$
and $d$ to be the (normalized) Hamming metric $d^{H}$ on $\Sym\left(n\right)$.
This section deals with deforming a function $f\colon\Gamma\rightarrow G$
into a symmetric function $f'\colon\Gamma\rightarrow G$ (see Definition
\ref{def:symmetric-function}). This is achieved by Proposition \ref{prop:symmetrization-in-general}
and is used in Section \ref{sec:amenable-stable} via Proposition
\ref{prop:U-symmetrization}.

Throughout the section, the notation $\trieq$ indicates the use of
the triangle inequality. We use the bi-invariance of the metric $d\left(\cdot,\cdot\right)$
freely, without further explanation.

Let $\Gamma$ be a group and take functions $f,h:\Gamma\rightarrow G$.
The \emph{uniform local defect} of $f$ is
\[
\defect_{\infty}\left(f\right)=\sup\left\{ d\left(f\left(\gamma_{1}\gamma_{2}\right),f\left(\gamma_{1}\right)f\left(\gamma_{2}\right)\right)\mid\gamma_{1},\gamma_{2}\in\Gamma\right\} 
\]
and the \emph{uniform distance} between $f$ and $h$ is
\[
d_{\infty}\left(f,h\right)=\sup\left\{ d\left(f\left(\gamma\right),h\left(\gamma\right)\right)\mid\gamma\in\Gamma\right\} \,\,\text{.}
\]
If $\Gamma$ is a discrete amenable group equipped with a left- or
right- invariant measure $m$, then the \emph{mean local defect} of
$f$ is
\[
\defect_{1}\left(f\right)=\intint d\left(f\left(\gamma_{1}\gamma_{2}\right),f\left(\gamma_{1}\right)f\left(\gamma_{2}\right)\right)\dm\left(\gamma_{1}\right)\dm\left(\gamma_{2}\right)
\]
and the \emph{mean distance} between $f$ and $h$ is
\[
d_{1}\left(f,h\right)=\int d\left(f\left(\gamma\right),h\left(\gamma\right)\right)\dm\left(\gamma\right)\,\,\text{.}
\]

We begin with the following lemma, that says that if two functions
are close together then their local defects are nearly the same.
\begin{lem}
\label{lem:proximity-implies-similar-local-defect}Let $\Gamma$ be
a group and let $f,f'\colon\Gamma\rightarrow G$ be functions. Then
$\defect_{\infty}\left(f'\right)\leq3d_{\infty}\left(f,f'\right)+\defect_{\infty}\left(f\right)$.

Furthermore, if $\Gamma$ is a discrete amenable group equipped with
a right-invariant measure $m$, then $\defect_{1}\left(f'\right)\leq3d_{1}\left(f,f'\right)+\defect_{1}\left(f\right)$.
\end{lem}

\begin{proof}
For $\gamma_{1},\gamma_{2}\in\Gamma$, we have
\begin{align*}
d\left(f'\left(\gamma_{1}\gamma_{2}\right),f'\left(\gamma_{1}\right)f'\left(\gamma_{2}\right)\right)\leq & d\left(f'\left(\gamma_{1}\gamma_{2}\right),f\left(\gamma_{1}\gamma_{2}\right)\right)+d\left(f\left(\gamma_{1}\gamma_{2}\right),f\left(\gamma_{1}\right)f\left(\gamma_{2}\right)\right) & \text{by \ensuremath{\trieq}}\\
 & +\underbrace{d\left(f\left(\gamma_{1}\right)f\left(\gamma_{2}\right),f\left(\gamma_{1}\right)f'\left(\gamma_{2}\right)\right)}_{=d\left(f\left(\gamma_{2}\right),f'\left(\gamma_{2}\right)\right)}+\underbrace{d\left(f\left(\gamma_{1}\right)f'\left(\gamma_{2}\right),f'\left(\gamma_{1}\right)f'\left(\gamma_{2}\right)\right)}_{=d\left(f\left(\gamma_{1}\right),f'\left(\gamma_{1}\right)\right)}\,\,\text{.}
\end{align*}
Hence,
\[
\sup\left\{ d\left(f'\left(\gamma_{1}\gamma_{2}\right),f'\left(\gamma_{1}\right)f'\left(\gamma_{2}\right)\right)\mid\gamma_{1},\gamma_{2}\in\Gamma\right\} \leq3d_{\infty}\left(f,f'\right)+\defect_{\infty}\left(f\right)
\]
and if $m$ is a right-invariant measure on $\Gamma$, then 
\[
\intint d\left(f'\left(\gamma_{1}\gamma_{2}\right),f'\left(\gamma_{1}\right)f'\left(\gamma_{2}\right)\right)\dm\left(\gamma_{1}\right)\dm\left(\gamma_{2}\right)\leq3d_{1}\left(f,f'\right)+\defect_{1}\left(f\right)\,\,.
\]
\end{proof}
We turn to the task of symmetrizing a given function $f\colon\Gamma\rightarrow G$.
Note that 
\begin{equation}
d\left(f\left(1_{\Gamma}\right),1_{G}\right)=d\left(f\left(1_{\Gamma}\right)f\left(1_{\Gamma}\right),f\left(1_{\Gamma}\right)\right)\leq\defect_{\infty}\left(f\right)\label{eq:symmetrization-almost-id}
\end{equation}
and thus for $\gamma\in\Gamma$,
\begin{align}
d\left(f\left(\gamma\right),\left(f\left(\gamma^{-1}\right)\right)^{-1}\right) & =d\left(f\left(\gamma\right)f\left(\gamma^{-1}\right),1_{G}\right)\nonumber \\
 & \leq d\left(f\left(\gamma\right)f\left(\gamma^{-1}\right),f\left(1_{\Gamma}\right)\right)+d\left(f\left(1_{\Gamma}\right),1_{G}\right) & \text{by \ensuremath{\trieq}}\nonumber \\
 & \leq2\defect_{\infty}\left(f\right)\,\,\text{.}\label{eq:symmetrization-almost-multiplicative}
\end{align}

The following is a natural attempt to produce a symmetric function
$f'\colon\Gamma\rightarrow G$ close to $f$.
\begin{itemize}
\item Let $\Gamma_{2}=\left\{ \gamma\in\Gamma\mid\gamma^{2}=1\right\} $
and fix a set $B$ containing exactly one of $\gamma$ and $\gamma^{-1}$
for each $\gamma\in\Gamma\setminus\Gamma_{2}$.
\item Set $f'\left(1_{\Gamma}\right)=1_{G}$.
\item For $\gamma\in B$, set $f'\left(\gamma\right)=f\left(\gamma\right)$
and $f'\left(\gamma^{-1}\right)=\left(f\left(\gamma\right)\right)^{-1}$.
\item For $\gamma\in\Gamma_{2}\setminus\left\{ 1\right\} $, let $f'\left(\gamma\right)$
be an order-two element of $G$ that is close to $f\left(\gamma\right)$.
\end{itemize}
The function $f'\colon\Gamma\to G$ is symmetric by construction,
and we would like to bound $d_{\infty}\left(f,f'\right)$ and $d_{1}\left(f,f'\right)$.
We shall see that we can obtain good bounds if we can perform the
last step efficiently, that is, if approximate square roots in $G$
are close to square roots. First, we investigate approximate square
roots in the case $G=\Sym\left(n\right)$.
\begin{lem}
\label{lem:2-torsion-Sym}Let $\sigma\in\Sym\left(n\right)$. Then,
there is $\tau\in\Sym\left(n\right)$ such that $\tau^{2}=\id$ and
$d^{H}\left(\sigma,\tau\right)=d^{H}\left(\sigma^{2},\id\right)$. 
\end{lem}

\begin{proof}
Let $A=\left\{ x\in\left[n\right]\mid\sigma^{2}\left(x\right)=x\right\} $.
Note that the restriction $\sigma\mid_{A}$ is an involution $A\rightarrow A$.
Define 
\[
\tau\left(x\right)=\begin{cases}
\sigma\left(x\right) & x\in A\\
x & x\notin A
\end{cases}\,\,\text{.}
\]
Then $\tau^{2}={\rm id}$ and $d^{H}\left(\sigma,\tau\right)=1-\frac{|A|}{n}=d^{H}\left(\sigma^{2},\id\right)$.
\end{proof}
In other words, the lemma says that the cyclic group $C_{2}$ of order
two is stable w.r.t. $\Sym\left(n\right)$, and bounds the stability
rate.

Assume that there is a real number $M_{2}$ such that for every $g\in G$
there is $h\in G$ satisfying $h^{2}=1_{G}$ and $d\left(g,h\right)\leq M_{2}d\left(g^{2},1_{G}\right)$.
In the context of symmetrization of a function $\Gamma\rightarrow G$,
we would like $M_{2}$ to be small. Lemma \ref{lem:2-torsion-Sym}
implies that in the case $G=\Sym\left(n\right)$, we may take $M_{2}=1$.
\begin{prop}
\label{prop:symmetrization-in-general}With $M_{2}$ as above, write
$C=2\max\left\{ 1,M_{2}\right\} $. Let $\Gamma$ be a group and $f\colon\Gamma\rightarrow G$
a function. Then, there is a symmetric function $f'\colon\Gamma\rightarrow G$
such that $d_{\infty}\left(f,f'\right)\leq C\delta_{\infty}$ and
$\defect_{\infty}\left(f'\right)\leq\left(3C+1\right)\delta_{\infty}$,
where $\delta_{\infty}=\defect_{\infty}\left(f\right)$.

Furthermore, if $\Gamma$ is a discrete amenable group equipped with
a right-invariant inverse-invariant measure $m$, then for the same
function $f'$ we have $d_{1}\left(f,f'\right)\leq\left(C+1\right)\delta_{1}$
and $\defect_{1}\left(f,f'\right)\leq\left(3C+4\right)\delta_{1}$,
where $\delta_{1}=\defect_{1}\left(f\right)$.
\end{prop}

\begin{proof}
By Lemma \ref{lem:proximity-implies-similar-local-defect}, the bounds
on $\defect_{\infty}\left(f'\right)$ and $\defect_{1}\left(f'\right)$
follow from the bounds on $d_{\infty}\left(f,f'\right)$ and $d_{1}\left(f,f'\right)$,
respectively. We turn to the proof of the latter.

Let $\Gamma_{2}$ and $B$ be as in the discussion preceding Lemma
\ref{lem:2-torsion-Sym}. For each $\gamma\in\Gamma_{2}\setminus\left\{ 1_{\Gamma}\right\} $,
take $\tau\left(\gamma\right)\in G$ such that

\begin{equation}
d\left(f\left(\gamma\right),\tau\left(\gamma\right)\right)\le M_{2}d\left(\left(f\left(\gamma\right)\right)^{2},1_{G}\right)\,\,.\label{eq:symmetrization-tau-function}
\end{equation}
Define a function $f'\colon\Gamma\rightarrow G$ as follows:
\[
f'\left(\gamma\right)=\begin{cases}
1_{G} & \gamma=1_{\Gamma}\\
f\left(\gamma\right) & \gamma\in B\\
\left(f\left(\gamma^{-1}\right)\right){}^{-1} & \gamma\in\Gamma\backslash\left(\Gamma_{2}\cup B\right)\\
\tau\left(\gamma\right) & \gamma\in\Gamma_{2}\setminus\left\{ 1_{\Gamma}\right\} 
\end{cases}\,\,\text{.}
\]
Then $f'$ is symmetric by construction. By (\ref{eq:symmetrization-almost-id}),
we have
\begin{equation}
d\left(f\left(1_{\Gamma}\right),f'\left(1_{\Gamma}\right)\right)=d\left(f\left(1_{\Gamma}\right),1_{G}\right)\leq\delta_{\infty}\,\,\text{.}\label{eq:symmetrization-identitty-barely-moved}
\end{equation}
By (\ref{eq:symmetrization-almost-multiplicative}), for $\gamma\in\Gamma\setminus\left(\Gamma_{2}\cup B\right)$
we have

\[
d\left(f\left(\gamma\right),f'\left(\gamma\right)\right)=d\left(f\left(\gamma\right),\left(f\left(\gamma^{-1}\right)\right)^{-1}\right)\leq2\delta_{\infty}\,\,\text{.}
\]
Finally, for $\gamma\in\Gamma_{2}\setminus\left\{ 1_{\Gamma}\right\} $
we have
\begin{align*}
d\left(f\left(\gamma\right),f'\left(\gamma\right)\right) & =d\left(f\left(\gamma\right),\tau\left(\gamma\right)\right) & \text{}\\
 & \leq M_{2}d\left(\left(f\left(\gamma\right)\right)^{2},1_{G}\right) & \text{by (\ref{eq:symmetrization-tau-function})}\\
 & \leq M_{2}\cdot\left(d\left(\left(f\left(\gamma\right)\right)^{2},f\left(\gamma^{2}\right)\right)+d\left(f\left(1_{\Gamma}\right),1_{G}\right)\right) & \text{by \ensuremath{\trieq\ }and \ensuremath{\gamma^{2}}=\ensuremath{1_{\Gamma}}}\\
 & \leq2M_{2}\delta_{\infty}\,\,\text{.} & \text{by (\ref{eq:symmetrization-identitty-barely-moved})}
\end{align*}
This finishes the proof that $d_{\infty}\left(f',f\right)\leq C\delta_{\infty}$.

Now, assume that $m$ is a right-invariant inverse-invariant measure
on $\Gamma$. First, we show that $f$ approximately respects inverses
on average:
\begin{align}
\int d\left(f\left(\gamma\right)f\left(\gamma^{-1}\right),1_{G}\right)\dm\left(\gamma\right)= & \intint d\left(f\left(\gamma\right)f\left(\gamma^{-1}\right),1_{G}\right)\dm\left(\gamma'\right)\dm\left(\gamma\right)\nonumber \\
= & \intint d\left(f\left(\gamma'\right)f\left(\gamma\right)f\left(\gamma^{-1}\right),f\left(\gamma'\right)\right)\dm\left(\gamma'\right)\dm\left(\gamma\right)\nonumber \\
\leq & \intint d\left(f\left(\gamma'\right)f\left(\gamma\right)f\left(\gamma^{-1}\right),f\left(\gamma'\gamma\right)f\left(\gamma^{-1}\right)\right)\dm\left(\gamma'\right)\dm\left(\gamma\right) & \text{by \ensuremath{\trieq}}\nonumber \\
 & +\intint d\left(f\left(\gamma'\gamma\right)f\left(\gamma^{-1}\right),f\left(\gamma'\right)\right)\dm\left(\gamma'\right)\dm\left(\gamma\right)\nonumber \\
= & 2\delta_{1}\,\,\text{,}\label{eq:symmetrization-resepect-for-inverses-on-average}
\end{align}
where for the last equality, we manipulate the first term by cancelling
out $f\left(\gamma^{-1}\right)$, and in the second term we apply
$\gamma\mapsto\gamma^{-1}$ and then $\gamma'\mapsto\gamma'\gamma$.
Furthermore, $f$ approximately respects the identity elements in
the following sense (which is significant only when $\Gamma$ is finite): 

\begin{align}
m\left(\left\{ 1_{\Gamma}\right\} \right)d\left(f\left(1_{\Gamma}\right),1_{G}\right) & =\int_{\left\{ 1_{\Gamma}\right\} }\int d\left(f\left(1_{\Gamma}\right),1_{G}\right)\dm\left(\gamma'\right)\dm\left(\gamma\right)\nonumber \\
 & =\int_{\left\{ 1_{\Gamma}\right\} }\int d\left(f\left(\gamma'\right)f\left(1_{\Gamma}\right),f\left(\gamma'\cdot1_{\Gamma}\right)\right)\dm\left(\gamma'\right)\dm\left(\gamma\right)\nonumber \\
 & \leq\intint d\left(f\left(\gamma'\right)f\left(\gamma\right),f\left(\gamma'\gamma\right)\right)\dm\left(\gamma'\right)\dm\left(\gamma\right)\nonumber \\
 & =\delta_{1}\,\,\text{.}\label{eq:symmetrization-identity-on-average}
\end{align}
Now,
\begin{align}
\int_{\Gamma_{2}\setminus\left\{ 1_{\Gamma}\right\} }d\left(f\left(\gamma\right),\tau\left(\gamma\right)\right)\dm\left(\gamma\right) & \leq M_{2}\int_{\Gamma_{2}\setminus\left\{ 1_{\Gamma}\right\} }d\left(\left(f\left(\gamma\right)\right)^{2},\id\right)\dm\left(\gamma\right) & \text{by (\ref{eq:symmetrization-tau-function})}\nonumber \\
 & =M_{2}\int_{\Gamma_{2}\setminus\left\{ 1_{\Gamma}\right\} }d\left(f\left(\gamma\right)f\left(\gamma^{-1}\right),1_{G}\right)\dm\left(\gamma\right)\,\,\text{.} & \text{\ensuremath{\gamma=\gamma^{-1}} for \ensuremath{\gamma\in\Gamma_{2}}}\label{eq:symmetrization-average-dist-on-involutions}
\end{align}
Finally,
\begin{align*}
d_{1}\left(f,f'\right)= & \int d\left(f\left(\gamma\right),f'\left(\gamma\right)\right)\dm\left(\gamma\right)\\
= & m\left(\left\{ 1_{\Gamma}\right\} \right)d\left(f\left(1_{\Gamma}\right),1_{G}\right)+\int_{\Gamma\backslash\left(\Gamma_{2}\cup B\right)}d\left(f\left(\gamma\right)f\left(\gamma^{-1}\right),1_{G}\right)\dm\left(\gamma\right)\\
 & +\int_{\Gamma_{2}\setminus\left\{ 1_{\Gamma}\right\} }d\left(f\left(\gamma\right),\tau\left(\gamma\right)\right)\dm\left(\gamma\right)\\
\leq & \delta_{1}+\max\left\{ 1,M_{2}\right\} \int d\left(f\left(\gamma\right)f\left(\gamma^{-1}\right),1_{G}\right)\dm\left(\gamma\right) & \text{by (\ref{eq:symmetrization-identity-on-average}) and (\ref{eq:symmetrization-average-dist-on-involutions})}\\
\leq & \left(C+1\right)\delta_{1}\,\,\text{.} & \text{by (\ref{eq:symmetrization-resepect-for-inverses-on-average})}
\end{align*}
\end{proof}
\bibliographystyle{plain}
\bibliography{unistab}

\end{document}